\documentclass[a4paper,10pt]{article}
\usepackage[english]{babel}
\usepackage[utf8]{inputenc}
\usepackage[T1]{fontenc}
\usepackage{textcomp} 
\usepackage{amsmath}
\usepackage{enumerate}
\usepackage{amsfonts}
\usepackage{amsthm}
\usepackage{mathrsfs}
\usepackage{amssymb}
\usepackage{comment}
\usepackage{csquotes}
\usepackage{bm}
\usepackage{graphicx}
\usepackage{bbm}
\usepackage{array}
\usepackage{ragged2e}
\usepackage{mathtools}
\usepackage{xcolor}
\usepackage{lipsum}
\usepackage{tabularx}
\usepackage[mathscr]{euscript}
\usepackage[colorlinks=true,linkcolor=blue,urlcolor=blue]{hyperref}

\DeclareSymbolFont{rsfs}{U}{rsfs}{m}{n}
\DeclareSymbolFontAlphabet{\mathscrsfs}{rsfs}

\newcommand{\keywords}[1]{\par\noindent{\small\textbf{Keywords:} #1}}
\newcommand{\subjclass}[1]{\par\noindent{\small\textbf{AMS 2010 Subject Classification:} #1}}

\theoremstyle{definition}
\newtheorem{Def}{Definition}[section]

\theoremstyle{plain}
\newtheorem{Prop}[Def]{Proposition}
\newtheorem{thm}[Def]{Theorem}

\newcommand{\footremember}[2]{%
	\footnote{#2}
	\newcounter{#1}
	\setcounter{#1}{\value{footnote}}%
}
\newcommand{\footrecall}[1]{%
	\footnotemark[\value{#1}]%
}

\newcommand{\E}{E}

\newcommand{\Var}{\mathrm{Var}}
\newcommand{\Cov}{\mathrm{Cov}}

\newcommand{\dTV}{d_{\mathrm{TV}}}
\newcommand{\ip}[2]{\left\langle #1,#2\right\rangle}
\newcommand{\norm}[1]{\left\lVert #1\right\rVert}
\newcommand{\I}{I}
\newcommand{\kappaFour}{\kappa_{4}}

\newcolumntype{M}[1]{>{\raggedright}m{#1}}

\renewcommand{\epsilon}{\varepsilon}

\usepackage{graphicx}
\usepackage{pgffor}
\usepackage{caption}  

\begin{document}

\title{Mathematical research with GPT-5: a Malliavin-Stein experiment}
\author{
  Charles-Philippe Diez\footremember{uni}{University of Luxembourg}%
  \and
 Lu\'{\i}s da Maia\footrecall{uni}%
  \and 
  Ivan Nourdin\footrecall{uni}%
}

\date{}

\maketitle

\begin{abstract}
On August 20, 2025, GPT-5 was reported to have solved an open problem in convex optimization. Motivated by this episode, we conducted a controlled experiment in the Malliavin--Stein framework for central limit theorems. Our objective was to assess whether GPT-5 could go beyond known results by extending a \emph{qualitative} fourth-moment theorem to a \emph{quantitative} formulation with explicit convergence rates, both in the Gaussian and in the Poisson settings. To the best of our knowledge, the derivation of such quantitative rates had remained an open problem, in the sense that it had never been addressed in the existing literature. The present paper documents this experiment, presents the results obtained, and discusses their broader implications.
\end{abstract}



\section{Introduction}

The starting point of this study is a post by Sébastien Bubeck \cite{Bubeck2025Thread} on X (Aug.~20, 2025), reporting that GPT-5 Pro had solved an open problem in convex optimization by improving a known bound from $1/L$ to $1.5/L$ within minutes. Beyond the claim and two screenshots (one showing the prompt, the other the AI-generated proof), no further details were provided regarding the methodology.

\medskip

Bubeck's post attracted considerable attention, particularly on social media. Many non-specialists perceived it as a historic moment, even a striking demonstration of the power of AI, now seemingly able to compete with mathematicians. The reaction of mathematicians and researchers in the field was, however, more nuanced. Among the most notable comments was that of Ernest Ryu \cite{Ryu}, an expert in convex optimization, who placed the experiment back into context. According to him, the demonstration proposed by GPT-5 relied mainly on a well-known ingredient, Nesterov’s Theorem, already familiar to specialists. In his view, an experienced researcher could have obtained an equivalent result within a few hours of work.

\medskip

Motivated by this episode, we designed a small and controlled experiment in an area we know very well: the Malliavin–Stein method \cite{NourdinPeccati2012}, a powerful tool in probability theory to study convergence towards the normal distribution. This method was introduced almost twenty years ago by the fourth-named author together with Giovanni Peccati. It combines two complementary ideas. The Stein method makes it possible to test whether a random object converges to the normal law and, importantly, to measure the speed of this convergence. The Malliavin calculus, on the other hand, provides a kind of differential framework for random variables in stochastic analysis, especially on Gaussian and Poisson spaces.
By bringing these two tools together, the Malliavin–Stein method not only shows that convergence takes place, but also gives explicit rates of convergence in settings where Malliavin calculus can be applied.

\medskip

Concretely, we started from a recent theorem by Basse-O’Connor, Kramer-Bang, and Svendsen  \cite{basse}, which established a \emph{qualitative} result (proving that a certain sequence of probabilistic objects converges) but without specifying the speed of convergence. We then asked GPT-5 to go further and transform this qualitative result into a \emph{quantitative} one, that is, to provide an explicit convergence rate. To the best of our knowledge, no published solution to this precise problem existed until today.

\medskip

After this very rough description of the mathematical content, aimed mainly at non-mathematicians, we invite readers who are not specialists (or who are not primarily concerned with the mathematical results) to proceed directly to Section~\ref{gpt}. For the others, we now briefly recall the context of our study and outline the results we obtained, with more precision.

\medskip

The classical \emph{fourth moment theorem} of Nualart and Peccati \cite{nualart-peccati-2005} states that, for a sequence of normalized multiple Wiener--Itô integrals of fixed order, convergence of the fourth moment to three is equivalent to convergence in distribution to $\mathcal N(0,1)$. This principle underlies numerous applications, most notably in establishing central limit theorems for functionals of infinite-dimensional Gaussian fields. Within the Malliavin--Stein framework, quantitative refinements can be expressed as bounds on the distance to Gaussianity in terms of the fourth cumulant; see \cite{NourdinPeccati2012}.

\medskip

Building on \cite{basse}, in which the authors established a \emph{qualitative} fourth moment theorem for \emph{sums of two} multiple Wiener--Itô integrals of orders $p$ and $q$ such that $p+q$ is odd, we make use of GPT-5 to obtain a \emph{quantitative} counterpart in total variation. We also provide a Poisson analogue. In the Poisson setting, mixed odd moments such as $\E[X^3Y]$ need not vanish when $X$ and $Y$ are multiple Wiener--Itô integrals of different parities, so we identify sufficient conditions to establish a similar type of result. The theorem still holds, and we exhibit a counterexample showing these conditions are essentially sharp. It is important to note that, to the best of our knowledge, neither the Gaussian refinement nor the Poisson analogue had previously appeared in the literature.

\medskip

Before turning to broader reflections on this unusual and somewhat disorienting AI-assisted workflow for mathematicians, we first present our mathematical contributions in a self-contained and usual manner, disregarding how they were obtained. The discussion of our GPT-5 protocol and what we believe to be its implications for research and doctoral training is deferred to Section~\ref{gpt}.

\medskip

\noindent\textit{Reading guide.} Section~\ref{gauss} recalls the necessary background and states the main quantitative result we have obtained in the Gaussian setting. Section~\ref{poisson} is devoted to the Poisson extension and its limitations. Section~\ref{gpt} documents the GPT-5 experiment, and Section~\ref{sec:reflections} offers ethical and educational reflections. The paper concludes with two appendices, reproducing the discussions we had with GPT-5 in the Gaussian and Poisson cases.

\section{Gaussian analysis and Wiener chaos}
\label{gauss}

This section presents the results we obtained in the Gaussian setting.

\subsection{Preliminaries}

Here, we recall the main tools and results concerning Gaussian analysis and Wiener chaos that will be needed later in the paper. 
Our aim is not to be exhaustive but to provide the essential background. For further details and complete proofs, we refer the reader to \cite{NourdinPeccati2012,Nualart2006}.

\subsubsection{Wiener space and multiple integrals}

Let $(E, \mathcal E, \mu)$ be a measurable space and $H=L^2(E,\mu)$ the associated Hilbert space. Let $W=\{W(h):h\in H\}$ be an isonormal Gaussian process. For $m\ge 1$, the $m$-th multiple Wiener-It\^o integral of a kernel $f\in L^2_s(\mu^m)$ (square integrable and symmetric in its arguments) is denoted $I_m(f)$ and satisfies the isometry
\[
E[I_m(f)I_m(g)] = m! \langle f,g\rangle_{L^2(\mu^m)}.
\]
The construction of $I_m$ starts by defining it on simple tensors $f=\mathbf 1_{A_1}\otimes\cdots\otimes \mathbf 1_{A_m}$ with disjoint sets $A_i$, extends linearly, and finally proceeds by $L^2$-density and symmetrization. 
The collection of all such integrals spans the $m$-th Wiener chaos.

\subsubsection{Orthogonality, contractions and product formula}

The Wiener chaoses are mutually orthogonal: if $F=I_p(f)$ and $G=I_q(g)$ with $p\ne q$, then $\E[FG]=0$.
For $f\in L^2_s(\mu^p)$ and $g\in L^2_s(\mu^q)$ and an integer $0\le r\le \min(p,q)$, the $r$-th contraction $f\otimes_r g$ belongs to $L^2(\mu^{p+q-2r})$ and is defined by contracting $r$ coordinates:
\begin{align}
(f\otimes_r g)(x_1,\dots,x_{p-r},y_1,\dots,y_{q-r}) &= \int_{E^r} f(x_1,\dots,x_{p-r},z_1,\dots,z_r)\notag \\
&\quad\times g(y_1,\dots,y_{q-r},z_1,\dots,z_r) d\mu(z_1)\cdots d\mu(z_r).\label{eq:product}
\end{align}
The symmetrized version is denoted $f\widetilde\otimes_r g$. 
The product formula states
\begin{equation}
I_p(f)\,I_q(g) = \sum_{r=0}^{\min(p,q)} r!\binom{p}{r}\binom{q}{r}\,I_{p+q-2r}(f\widetilde\otimes_r g).
\end{equation}

\subsubsection{Malliavin--Stein bound in total variation}
Let $\mathbb{D}^{1,2}$ be the domain of the Malliavin derivative $D$ and $L$ be the associated Ornstein--Uhlenbeck generator (with pseudo-inverse $L^{-1}$). For any centered, unit-variance $F\in \mathbb{D}^{1,2}$,
\begin{equation}\label{eq:MS-TV}
\dTV(F,\mathcal N(0,1))\ \le\ 2\,E\big|1-\ip{DF}{-DL^{-1}F}_{H}\big|
\ \le\ 2\sqrt{\Var\!\big(\ip{DF}{-DL^{-1}F}_{H}\big)}.
\end{equation}
This is standard in the Malliavin--Stein method; see, e.g., \cite[Th.\,5.2]{Nourdin}.
Furthermore, if $F=\I_m(f)$ then 
\begin{equation}\label{eq:Dt}
D_tF=m\,\I_{m-1}\!\big(f(\cdot,t)\big),\qquad
E\big[F^2\big]=m!\,\norm{f}^{2}_{H^{\otimes m}},
\end{equation}
and
$-DL^{-1}F=\frac1m DF$ by the chaos action of $L^{-1}$.

\subsection{Our main result}

\begin{thm}[Quantitative two-chaos fourth moment theorem]\label{thm:main}
For integers $p\neq q$, with $p$ odd, $q$ even, let
\begin{equation}
X = I_p(f), \qquad Y = I_q(g), \qquad Z = X+Y,
\end{equation}
satisfying $E[Z^2]=1$. Write $\kappa_4(Z)=\mathbb E[Z^4]-3$.
Then, we have
\begin{equation}
d_{\mathrm{TV}}(Z,N(0,1)) \;\le\; \sqrt{6\,\kappa_4(Z)},
\end{equation}
with $d_{\mathrm{TV}}$ the total variation distance.
In particular, if $Z_n = I_p(f_n) + I_q(g_n)$ with $E[Z_n^2]=1$ and $\kappa_4(Z_n)\to 0$, then
\[
d_{\mathrm{TV}}(Z_n,N(0,1)) \;\longrightarrow\;0.
\]
\end{thm}
\noindent
{\bf Proof}.
We split the argument into four steps. Throughout, set $\sigma_p^2=\E[Y^2]$, $\sigma_q^2=\E[Z^2]$; then $\sigma_p^2+\sigma_q^2=1$.

\medskip
\noindent
{\it Step 1: Malliavin--Stein reduction}.
Applying \eqref{eq:MS-TV} with $F=Z=X+Y$ and using $-DL^{-1}I_m(f)=\frac1m DI_m(f)$,
\begin{align}
\ip{DZ}{-DL^{-1}Z}
&= \frac1p\norm{DX}^{2}+\frac1q\norm{DY}^{2}+\Big(\frac1p+\frac1q\Big)\ip{DX}{DY}.
\end{align}
Define the centered pieces
\begin{equation}\label{eq:ApAqT}
A_p:=\sigma_p^2-\frac1p\norm{DX}^{2},\qquad
A_q:=\sigma_q^2-\frac1q\norm{DY}^{2},\qquad
T:=\Big(\frac1p+\frac1q\Big)\ip{DX}{DY}.
\end{equation}
Since $\sigma_p^2+\sigma_q^2=1$ and $\E\ip{DX}{DY}=0$ (orthogonality of different chaoses),
\[
1-\ip{DZ}{-DL^{-1}Z}=A_p+A_q-T,
\]
and hence
\begin{equation}\label{eq:VarStein}
\Var\!\big(\ip{DZ}{-DL^{-1}Z}\big)
=\E\big[(A_p+A_q-T)^2\big]
\le 3\Big(\E[A_p^2]+\E[A_q^2]+\E[T^2]\Big).
\end{equation}

\medskip
\noindent
{\it Step 2: Single-chaos control of $A_p$ and $A_q$}.
For a fixed chaos $F=\I_m(f)$ with variance $\sigma^2$, the identity
\begin{equation}\label{eq:SingleChaosStein}
\E\Big(\sigma^{2}-\frac1m\norm{DF}^{2}\Big)^{2}\ \le\ \frac13\,\big(\E[F^{4}]-3\sigma^{4}\big)
=\frac13\,\kappaFour(F)
\end{equation}
is classical (see, e.g., \cite[(5.61)]{Nourdin}). Applying \eqref{eq:SingleChaosStein} with $F=X$ and $F=Y$ gives
\begin{equation}\label{eq:ApAqBounds}
\E[A_p^2]\le \frac13\,\kappaFour(X),\qquad
\E[A_q^2]\le \frac13\,\kappaFour(Y).
\end{equation}

\medskip
\noindent
{\it Step 3: Exact expansion of $\E\ip{DX}{DY}^{2}$ and comparison with $\Cov(X^{2},Y^{2})$}.
Write $m:=\min\{p,q\}$. Using \eqref{eq:Dt} and \eqref{eq:product} for each $t$, and integrating over $t$,
\begin{equation}\label{eq:DYDZ-expansion}
\ip{DX}{DY}\ =\ pq\sum_{s=1}^{m} (s-1)!\binom{p-1}{s-1}\binom{q-1}{s-1}\,
\I_{p+q-2s}\!\big(f\widetilde{\otimes}_s g\big).
\end{equation}
By isometry and orthogonality of chaoses,
\begin{equation}\label{eq:EDYDZ2}
\E\ip{DX}{DY}^{2}
= p^{2}q^{2}\sum_{s=1}^{m}\Big[(s-1)!\binom{p-1}{s-1}\binom{q-1}{s-1}\Big]^{2}\,
(p+q-2s)!\,\big\| f\widetilde{\otimes}_{s}g\big\|^{2}.
\end{equation}
On the other hand, by  \cite[(3.5)]{nourdin-rosinski},
\begin{equation}\label{eq:CovY2Z2}
\Cov(X^{2},Y^{2})
= \sum_{s=1}^{m}\underbrace{\Big[s!\binom{p}{s}\binom{q}{s}\Big]^{2}(p+q-2s)!\,\big\| f\widetilde{\otimes}_{s}g\big\|^{2}}_{=:W_s}
\;+\;\underbrace{p!\,q!\sum_{s=1}^{m}\binom{p}{s}\binom{q}{s}\,\norm{f\otimes_{s}g}^{2}}_{\ge 0}.
\end{equation}

Binomial identities yield
\[
\Big[pq\,(s-1)!\binom{p-1}{s-1}\binom{q-1}{s-1}\Big]^{2}(p+q-2s)!\,\big\|f\widetilde{\otimes}_{s}g\big\|^{2}
= s^{2}W_s.
\]
Summing over $s$ gives the \emph{exact identity}
\begin{equation}\label{eq:Exact-s2}
\E\ip{DX}{DY}^{2}\ =\ \sum_{s=1}^{m} s^{2}\,W_s,
\qquad\text{with}\quad W_s\ \text{as in \eqref{eq:CovY2Z2}.}
\end{equation}
Since $W_s\ge 0$, we immediately obtain the universal comparison
\begin{equation}\label{eq:CrossBound}
\E\ip{DX}{DY}^{2}\ \le\ m^{2}\sum_{s=1}^{m} W_s\ \le\ m^{2}\,\Cov(X^{2},Y^{2}).
\end{equation}
Consequently, from \eqref{eq:ApAqT},
\begin{equation}\label{eq:ET2}
\E[T^{2}] = \Big(\tfrac1p+\tfrac1q\Big)^{2}\,\E\ip{DX}{DY}^{2}
\ \le 4\,\Cov(X^{2},Y^{2}).
\end{equation}

\medskip
\noindent
{\it Step 4: Parity-driven fourth-cumulant decomposition and conclusion}.
For general square-integrable $U,V$, one has
\[
\kappaFour(U+V)=\kappaFour(U)+\kappaFour(V)+6\,\Cov(U^{2},V^{2})+4\,\E[U^{3}V]+4\,\E[UV^{3}].
\]
If $U=X=\I_p(f)$ with $p$ odd and $V=Y=\I_q(g)$ with $q$ even, then the mixed odd terms vanish:
\begin{equation}\label{eq:ParityVanish}
\E\big[X^{3}Y\big]=\E\big[XY^{3}\big]=0,
\end{equation}
because in the product formula for (say) $X^{3}Y$ no zero-th chaos can appear (parity mismatch prevents total order from being zero). Hence
\begin{equation}\label{eq:CumulantParity}
\kappaFour(Z)=\kappaFour(X)+\kappaFour(Y)+6\,\Cov(X^{2},Y^{2}),
\end{equation}
with all three terms on the right nonnegative (fixed-chaos fourth cumulants are nonnegative, and each summand in \eqref{eq:CovY2Z2} is nonnegative). In particular,
\begin{equation}\label{eq:CovBoundByKappa}
\Cov(X^{2},Y^{2})\ \le\ \frac{\kappaFour(Z)}{6}\qquad\text{and}\qquad
\kappaFour(X)+\kappaFour(Y)\ \le\ \kappaFour(Z).
\end{equation}

Now combining \eqref{eq:VarStein}, \eqref{eq:ApAqBounds}, \eqref{eq:ET2} and \eqref{eq:CovBoundByKappa}
yields
\[
\Var\!\big(\ip{DZ}{-DL^{-1}Z}\big)\ \le\ \frac32\,\kappaFour(Z).
\]
Plugging this into  (\ref{eq:MS-TV}) implies the desired conclusion.
\qed

\section{Poisson framework}\label{poisson}

In this section, we aim to establish the main results in the Poisson framework, in close analogy with the Gaussian case. 

\subsection{Preliminaries}

First, we briefly recall the basic setup and notations in the Poisson space.

\subsubsection{Multiple Poisson--It\^o integrals}

Let $\eta$ be a Poisson random measure on $(E,\mathcal E)$ with control $\mu$, and let $\hat\eta=\eta-\mu$ be its compensated version. For $m\ge 1$ and $f\in L^2_s(\mu^m)$, the multiple Poisson integral $I_m^\eta(f)$ is defined by
\begin{align*}
I_m^\eta(f) &= \int_{E^m} f(x_1,\dots,x_m)\, \hat\eta(dx_1)\cdots\hat\eta(dx_m).
\end{align*}
The collection of all such integrals spans the $m$-th Poisson chaos, denoted $C_m$.
As in the Gaussian case, one has the isometry
\[
\E[I_m^\eta(f)I_m^\eta(g)] = m! \,\langle f,g\rangle_{L^2(\mu^m)},
\]
and Poisson chaoses are mutually orthogonal: if $F=I_p^\eta(f)$ and $G=I_q^\eta(g)$ with $p\ne q$, then $\E[FG]=0$.

\subsubsection{Fourth cumulant and positivity}\label{pos}

For $F=I_p^\eta(f)$, the fourth cumulant satisfies $\kappa_4(F)=\E[F^4]-3\E[F^2]^2\ge 0$, see \cite[(2.5)]{poisson}. Moreover, if $F\in C_p$ and $G\in C_q$, then $\mathrm{Cov}(F^2,G^2)\ge 0$ as a consequence of \cite[(2.4)]{poisson}. These positivity properties play a crucial role in fourth-moment theorems on the Poisson space.

\subsection{A Poisson counterpart of Theorem \ref{thm:main}}

We now present the result we obtained as a Poisson analogue of Theorem~\ref{thm:main}.

\medskip

For two given integers $p\neq q$, let
\begin{equation}\label{assump}
X_n = I_p^\eta(f_n), \qquad Y_n = I_q^\eta(g_n), \qquad Z_n = X_n+Y_n, \qquad \E[Z_n^2]=1.
\end{equation}
By orthogonality, $\E[X_nY_n]=0$. 

\begin{thm}[Fourth moment theorem under vanishing odd moments]\label{thm:positive}
In addition to (\ref{assump}), assume that
\begin{equation}\label{assump2}
\E[X_n^3Y_n] \to 0, \qquad \E[X_nY_n^3]\to 0.
\end{equation}
Then
\[
\E[Z_n^4]\to 3 \quad\Longrightarrow\quad Z_n \Rightarrow \mathcal N(0,1).
\]
\end{thm}

\begin{proof}
We compute the fourth cumulant:
\[
\kappa_4(Z_n) = \kappa_4(X_n)+\kappa_4(Y_n)+6\,\mathrm{Cov}(X_n^2,Y_n^2)+4 E[X_n^3Y_n]+4 E[X_nY_n^3].
\]
On the Poisson space, each individual fourth cumulant is nonnegative and $\mathrm{Cov}(X_n^2,Y_n^2)\ge 0$, see Section \ref{pos}. By assumption, the mixed odd terms vanish in the limit. Therefore, if $\E[Z_n^4]\to 3$, we obtain $\kappa_4(Z_n)\to 0$, which forces $\kappa_4(X_n)\to 0$ and $\kappa_4(Y_n)\to 0$. 

The fourth moment theorem on the Poisson chaos (see \cite[Corollary 1.3]{poisson}) then yields $X_n\Rightarrow \mathcal N(0,\sigma_p^2)$ and $Y_n\Rightarrow \mathcal N(0,\sigma_q^2)$ with $\sigma_p^2+\sigma_q^2=1$ (possibly along a subsequence). Since $\E[X_nY_n]=0$, the Peccati--Tudor type theorem for Poisson chaoses (see \cite[Corollary 1.8]{poisson}) implies $(X_n,Y_n)\Rightarrow (G_p,G_q)$ with $G_p,G_q$ independent Gaussian variables. Hence $Z_n=X_n+Y_n\Rightarrow \mathcal N(0,1)$. 
\end{proof}

\subsection{A counterexample when (\ref{assump2}) is not satisfied}

Consider the following particular case: take a measurable set $A$ with $\mu(A)=1$, and set
\[
U := I_1^\eta(\mathbf 1_A) = N_A - 1,
\]
where $N_A$ is Poisson distributed with mean $1$, and
\[
V := I_2^\eta(\mathbf 1_A^{\otimes 2}) = (N_A - 1)^2 - N_A.
\]
Clearly, $U\in C_1$ and $V\in C_2$, with $E[U]=E[V]=0$, $\mathrm{Var}(U)=1$, $\mathrm{Var}(V)=2$, and $E[UV]=0$.

For $\alpha\in\mathbb R$, define
\[
S_\alpha := c(\alpha)(U + \alpha V), \qquad c(\alpha) := \frac{1}{\sqrt{1+2\alpha^2}},
\]
so that $\mathrm{Var}(S_\alpha)=1$.

\begin{Prop}\label{thm:counterexample}
There exists $\alpha_*\in\mathbb R$ such that the random variable $S_{\alpha_*}$ satisfies
\[
E[S_{\alpha_*}^2]=1, \qquad E[S_{\alpha_*}^4]=3,
\]
while $S_{\alpha_*}$ is not Gaussian. In fact, $E[S_{\alpha_*}^3]\neq 0$.
\end{Prop}
\noindent
{\it Proof of Proposition \ref{thm:counterexample}}. 
It is divided into three steps.

\bigskip

\noindent
{\it Moments of $U$ and $V$}.
Recall $U=N_A-1$ with $N_A\sim\mathrm{Poi}(1)$. Then the centered moments of $U$ are
\[
E[U^2]=1, \quad E[U^3]=1, \quad E[U^4]=4.
\]
Similarly, $V=(N_A-1)^2-N_A$. A straightforward computation yields
\[
E[V^2]=2, \quad E[U V]=0.
\]
Further mixed moments can be obtained by direct expansion in terms of $N_A$ (or via the explicit Charlier polynomial representation). One finds
\begin{align*}
E[U^2 V] &= 6, & E[U^3 V] &= 6, & E[UV^2] &= 12, \\
E[U^2 V^2] &= 18, & E[UV^3] &= 56, & E[V^3] &= 12, \\
E[V^4] &= 212.
\end{align*}

\bigskip

\noindent
{\it Fourth moment of $S_\alpha$}.
We expand
\[
E[S_\alpha^4] = c(\alpha)^4 E[(U+\alpha V)^4].
\]
Using the values above,
\[
E[S_\alpha^4] = \frac{4 + 24\alpha + 108\alpha^2 + 224\alpha^3 + 212\alpha^4}{(1+2\alpha^2)^2}.
\]
Setting $E[S_\alpha^4]=3$ yields the quartic equation
\[
200\alpha^4 + 224\alpha^3 + 96\alpha^2 + 24\alpha + 1 = 0.
\]
This equation has a real solution $\alpha_*\approx -0.050832$. For this value, we have
\[
E[S_{\alpha_*}^2]=1, \qquad E[S_{\alpha_*}^4]=3.
\]

\bigskip

\noindent
{\it Third moment of $S_{\alpha_*}$}.
Expanding and Inserting the values of the mixed moments gives
\[
E[S_\alpha^3] = c(\alpha)^3 (1+6\alpha+12\alpha^2+12\alpha^3).
\]
At $\alpha=\alpha_*$ this equals approximately $0.719\neq 0$.
Thus $S_{\alpha_*}$ has variance one, fourth moment equal to three, and nonzero third moment. Consequently, its distribution cannot be Gaussian.
\qed

\bigskip

As a consequence, the conclusion of Theorem \ref{thm:positive} may fail without assumption (\ref{assump2}).

\section{GPT-5 as a research assistant}\label{gpt}

As mentioned in the introduction, we asked GPT-5 to turn the limit theorem proved in \cite{basse} into a quantitative one, by deriving explicit bounds on the total variation distance to the Gaussian law. To the best of our knowledge, this problem was open. Not in the sense of being particularly difficult, but simply because it had not previously been investigated.

\medskip

We now describe the process we followed in detail. 

\subsection{Protocol followed}

\subsubsection{Gaussian framework}

We started with the following initial prompt:

\begin{quote}
\texttt{Paper 2502.03596v1 establishes a qualitative fourth moment theorem for the sum of two Wiener--It\^o integrals of orders p and q, where p and q have different parities. Building on the Malliavin--Stein method (see 1203.4147v3 for details), could you derive a quantitative version for the total variation distance, with a convergence rate depending solely on the fourth cumulant of this sum?}  
\end{quote}

The first interaction (see Annex \ref{annex1} for the entire discussion) was strikingly effective. GPT-5 produced a generally correct statement, using the right tools and approach. However, it made a reasoning error (leading to a wrong expression for ${\rm Cov}(Y^2,Z^2)$) that could have invalidated the whole proof if left unchecked.

Noticing this, we then asked:
\begin{quote}
\texttt{Can you check your formula for ${\rm Cov}(Y^2,Z^2)$ and provide me with the details?}  
\end{quote}
It complied, giving the requested details. However, the formula was still incorrect, and the accompanying explanation was also wrong.  
We then pointed out the error more precisely:
\begin{quote}
\texttt{I think you are mistaken in claiming that $(p+q)!\|u\widetilde{\otimes}v\|^2 = p!q!\|u\|^2\|v\|^2$. Why should that be the case?}  
\end{quote}
It eventually admitted (which is not surprising, since by alignment it usually agrees with us) that the statement was false, but more importantly, it understood where the mistake came from. This was followed by a reasoning and a formula that, this time, were correct.

\medskip

Then, at our request, GPT-5 reformatted the result in the style of a research article, including an introduction, the presentation of our main theorem, its proof with all the details (correct this time!), and a bibliography. The exact prompt was:

\begin{quote}
\texttt{Turn this into a research paper ready for submission. Follow my style (see attached paper 0705.0570v4):\\ 
- start with an introduction giving some context, \\
- then present the main result, followed by a very detailed proof where no step is left out, \\
- finish with a complete bibliography. \\
The final document should be a LaTeX file that I can compile.}
\end{quote}

Finally, we asked it to add a concluding section containing possible extensions of the result that could be envisaged in future work.

\begin{quote}
\texttt{Can you add a ``Concluding Remarks'' section, where you summarize the main points and propose possible directions or extensions for future work?.}
\end{quote}

It complied and proposed a \emph{Concluding remarks} section, which ended with the following lines:

\begin{quotation}\small
Finally, one might ask whether the same approach can be adapted to other
Gaussian settings (e.g., Gaussian subordinated fields, Breuer–Major-type theorems) or even to
non-Gaussian frameworks where the Malliavin–Stein method has been successfully applied.
\end{quotation}

Building on this last suggestion (which is nothing extraordinary by the way, such extensions being quite natural in this context), we decided to continue our investigations and to explore an extension to the Poisson setting. 

\subsubsection{Poisson framework}

Since we found that the context window was already rather long and that this might possibly alter its performance (as an overload of information may reduce effectiveness), we opened a new session (see Annex \ref{annex2}) with the following short prompt:

\begin{quote}
\texttt{Here is a paper (2502.03596v1) proving a fourth-moment theorem for the sum of two Wiener--It\^o integrals with different parities. 
I would like you to extend it to the Poisson case, using the ideas contained in 1707.01889v2.}
\end{quote}

In this new session, GPT-5 quickly identified the structural difference with the Gaussian case: the mixed expectation $\E[X^3Y]$ does not necessarily vanish when $X$ and $Y$ are multiple Poisson integrals of different orders.  
On the other hand, it completely missed the fact that, just as in the Gaussian case, one still has ${\rm Cov}(X^2,Y^2)\geq 0$ in the Poisson case. We then tried to put it on track by asking:
\begin{quote}
\texttt{In paper 1707.01889v2, isn't there anything that could show that ${\rm Cov}(X^2,Y^2)$ is always positive?}
\end{quote}
But since the question we asked was open-ended, this was not enough to trigger the right idea. With great confidence, it replied: "short answer: no" and then gave an unconvincing explanation as to why.  

However, once we pointed out where to look:
\begin{quote}
\texttt{What about (2.4)?}
\end{quote}
it immediately understood how (2.4) indeed implied that ${\rm Cov}(X^2,Y^2)\geq 0$. It then reformulated its theorem to take this positivity into account, after we asked:
\begin{quote}
\texttt{So, could you give the new statement of the theorem that this implies?}
\end{quote}

Finally, at our request, it also produced a counterexample showing that, without the assumptions imposed in the theorem, the conclusion may fail. 

\subsubsection{Role of the AI}

To summarize, we can say that the role played by the AI was essentially that of an executor, responding to our successive prompts. Without us, it would have made a damaging error in the Gaussian case, and it would not have provided the most interesting result in the Poisson case, overlooking an essential property of covariance, which was in fact easily deducible from the results contained in the document we had provided.

\section{Some personal reflections}
\label{sec:reflections}

Overall, the experience of doing mathematics with GPT-5 was mixed. 
It felt very similar to working with a junior assistant at the beginning of a new project: exploring directions, formulating hypotheses, searching for counterexamples, and progressively adjusting statements. 
The AI showed a genuine ability to follow guided reasoning, to recognize its mistakes when pointed out, to propose new research directions, and to never take on the task. 
However, this only seems to support \emph{incremental} research, that is, producing new results that do not require genuinely new ideas but rather the ability to combine ideas coming from different sources. 
At first glance, this might appear useful for an exploratory phase, helping us save time. 
In practice, however, it was quite the opposite: we had to carefully verify everything produced by the AI and constantly guide it so that it could correct its mistakes. 

\medskip

The main risk we see with this technology, in its current state, is that it will almost certainly lead to a proliferation of incremental results produced with AI.
This could saturate the scientific landscape with technically correct but only moderately interesting contributions, making it harder for truly original work to stand out. 
The situation is reminiscent of other cultural domains already transformed by mass generative technologies: a flood of technically competent but uninspired outputs that dilutes attention and raises the noise level.

\medskip

We also foresee a second, more specific negative effect, concerning PhD students. 
Traditionally, when PhD students begin their dissertation, they are given a problem that is accessible but rich enough to help them become familiar with the tools, develop intuition, and learn to recognize what works and what does not. 
They typically read several papers, explore how a theory could be adapted, make mistakes, and eventually find their own path. 
This process, with all its difficulties, is part of what makes them independent researchers. 
If students rely too heavily on AI systems that can immediately generate technically correct but shallow arguments, they may lose essential opportunities to develop these fundamental skills. 
The danger is not only a loss of originality, but also a weakening of the very process of becoming a mathematician. 

In conclusion, we are still far from sharing the unreserved enthusiasm sparked by Bubeck’s post. 
Nevertheless, this development deserves close monitoring. 
The improvement over GPT-3.5/4 has been significant and achieved in a remarkably short time, which suggests that further advances are to be expected. 
Whether such progress could one day substantially displace the role of mathematicians remains an open question that only the future will tell.

\section*{Acknowledgements}
We gratefully acknowledge support from the Luxembourg National Research Fund (Grant O22/17372844/FraMStA).

\appendix 
\section{Annexes}

This appendix gathers supplementary material that illustrates the interaction protocol and the intermediate outputs produced during our experiment. The goal is to document the process in a transparent way and to offer visual evidence supporting the descriptions given in the main text.

\subsection{Screenshots in the Gaussian case}\label{annex1}

The following screenshots relate to the Gaussian framework discussed in Section~\ref{gauss}. 

\noindent
  \hspace*{-0.5em}%
 \!\!\!\!\!\!\!\!\!\!\!\!\!\!\!\!\!\!\!\!\!\!\!\!\!\!\!\!\!\!\!\!\!\!\!\!\!\!\!\!\!\!\!\!\!\!\!\!!\!\!\!\!\!\!\!\!\!\!\!\!\!\!
 \includegraphics[width=2\linewidth]{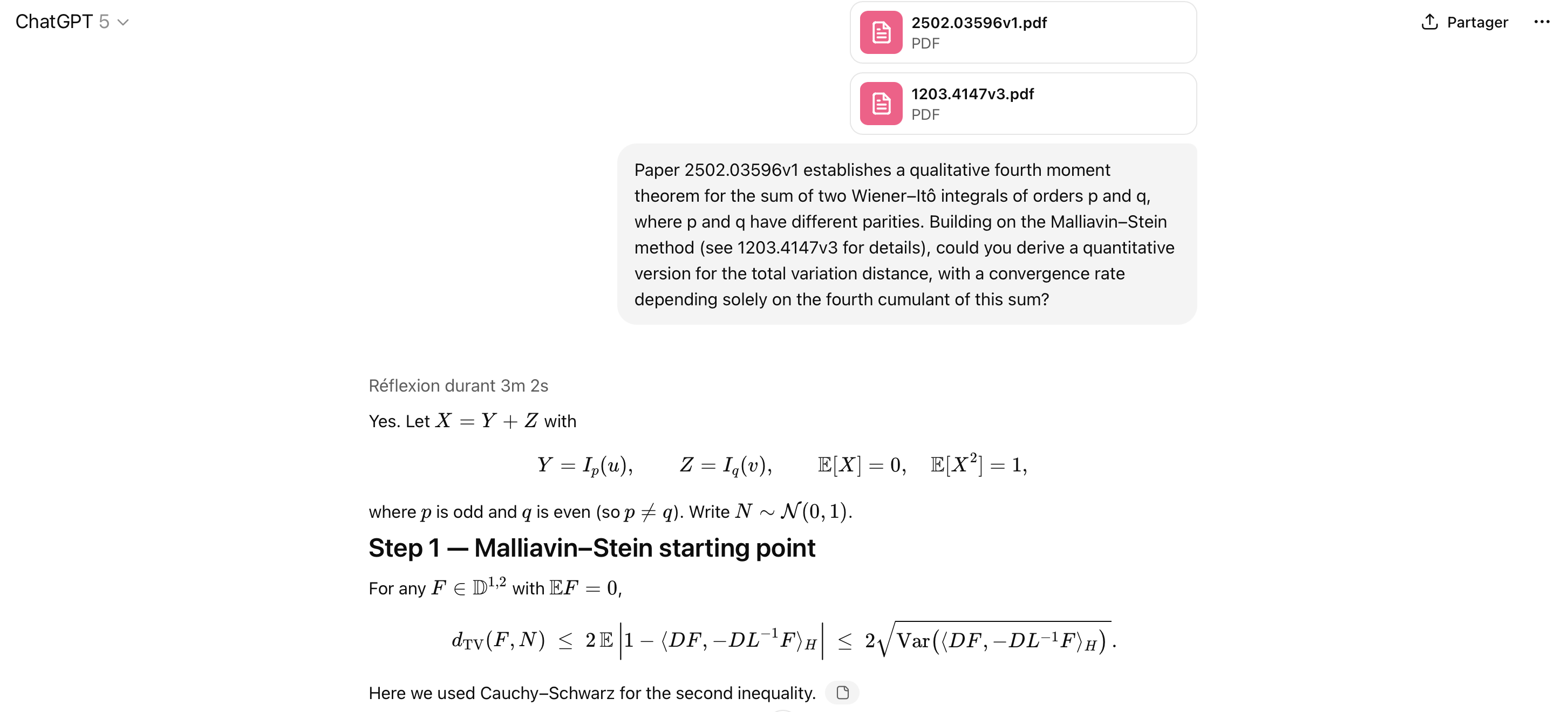}\par
  \medskip

\noindent
  \hspace*{-0.5em}%
 \!\!\!\!\!\!\!\!\!\!\!\!\!\!\!\!\!\!\!\!\!\!\!\!\!\!\!\!\!\!\!\!\!\!\!\!\!\!\!\!\!\!\!\!\!\!\!\!!\!\!\!\!\!\!\!\!\!\!\!\!\!\!
 \includegraphics[width=2\linewidth]{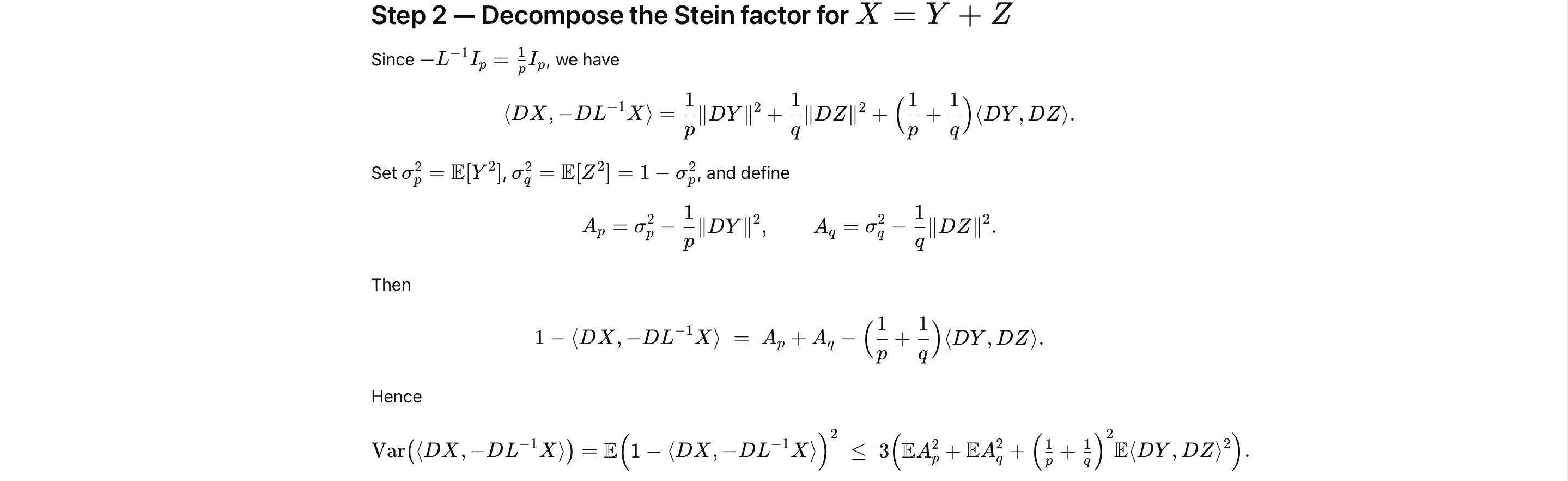}\par
  \medskip

\noindent
  \hspace*{-0.5em}%
 \!\!\!\!\!\!\!\!\!\!\!\!\!\!\!\!\!\!\!\!\!\!\!\!\!\!\!\!\!\!\!\!\!\!\!\!\!\!\!\!\!\!\!\!\!\!\!\!!\!\!\!\!\!\!\!\!\!\!\!\!\!\!
 \includegraphics[width=2\linewidth]{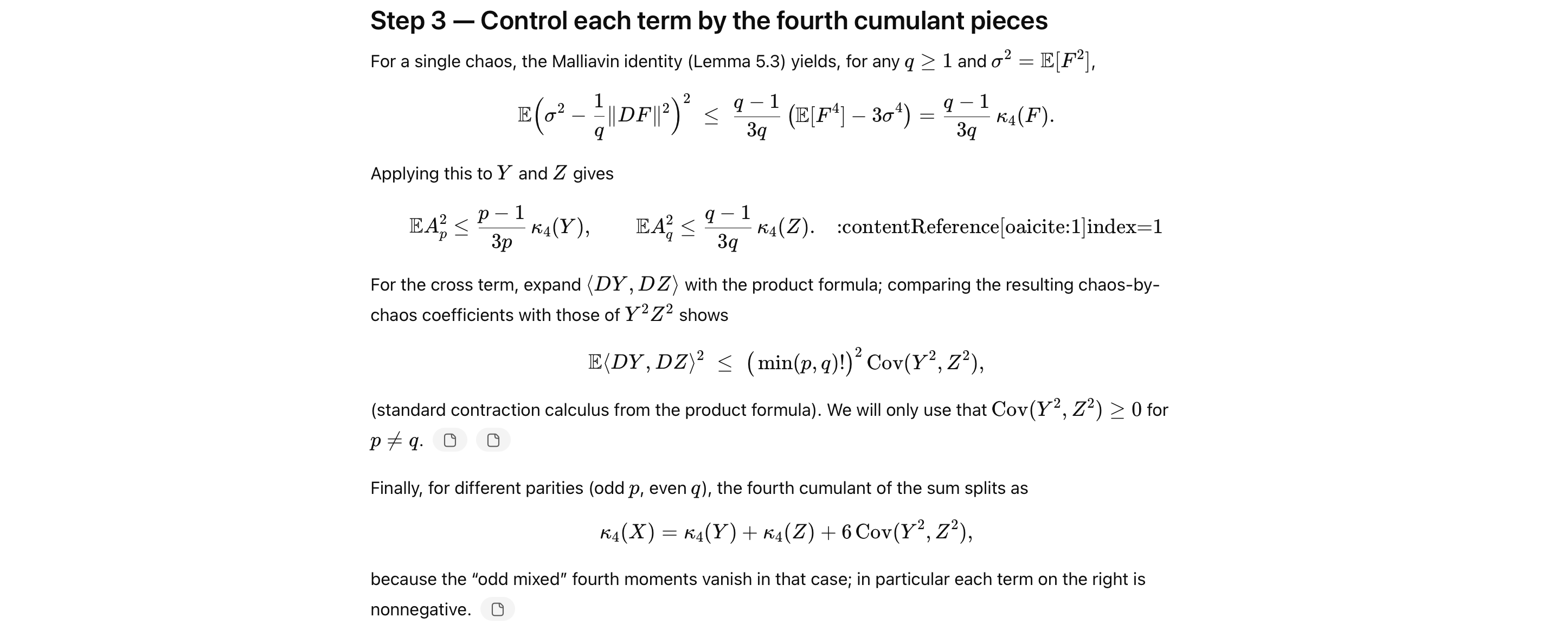}\par
  \medskip

\noindent
  \hspace*{-0.5em}%
 \!\!\!\!\!\!\!\!\!\!\!\!\!\!\!\!\!\!\!\!\!\!\!\!\!\!\!\!\!\!\!\!\!\!\!\!\!\!\!\!\!\!\!\!\!\!\!\!!\!\!\!\!\!\!\!\!\!\!\!\!\!\!
 \includegraphics[width=2\linewidth]{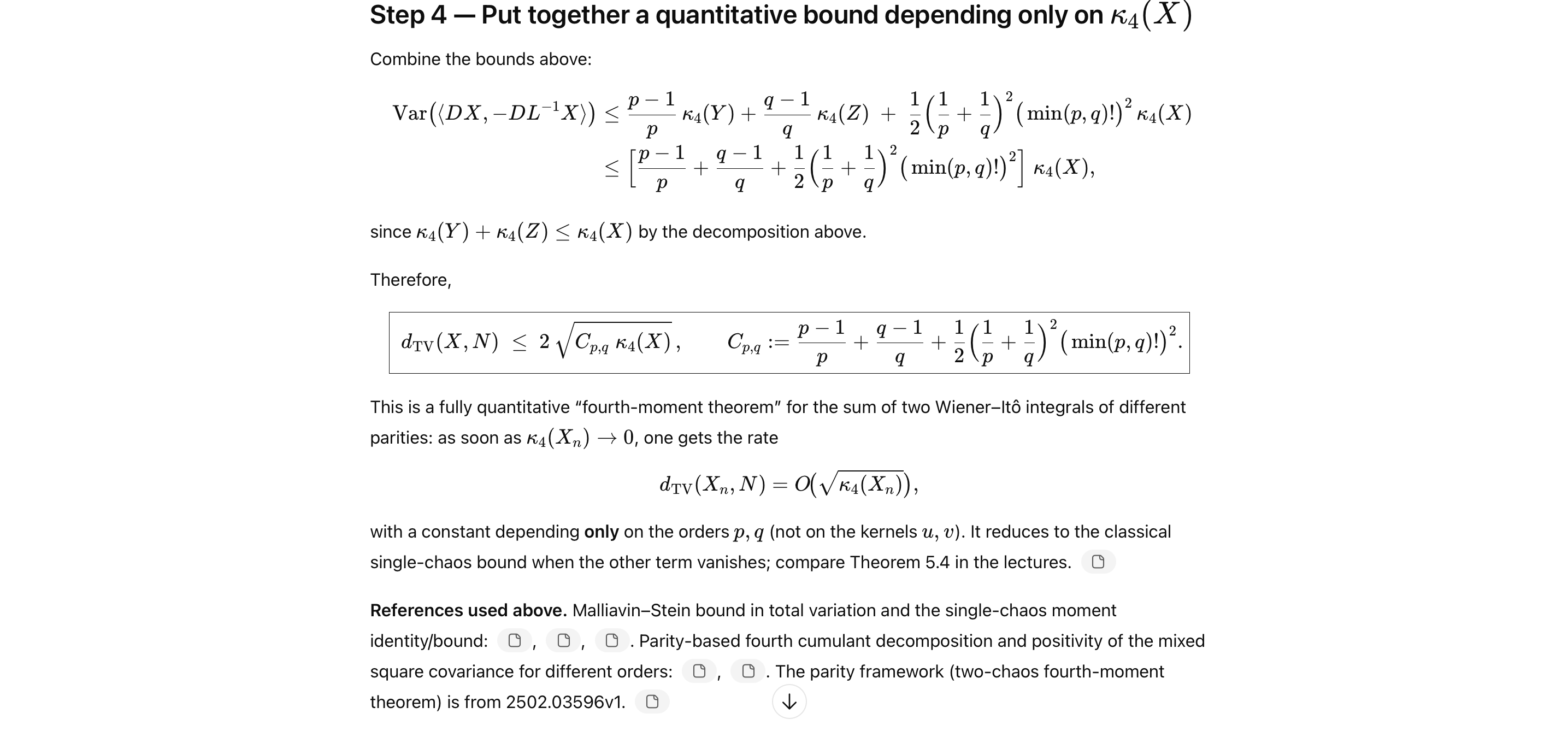}\par
  \medskip

\noindent
  \hspace*{-0.5em}%
 \!\!\!\!\!\!\!\!\!\!\!\!\!\!\!\!\!\!\!\!\!\!\!\!\!\!\!\!\!\!\!\!\!\!\!\!\!\!\!\!\!\!\!\!\!\!\!\!!\!\!\!\!\!\!\!\!\!\!\!\!\!\!
 \includegraphics[width=2\linewidth]{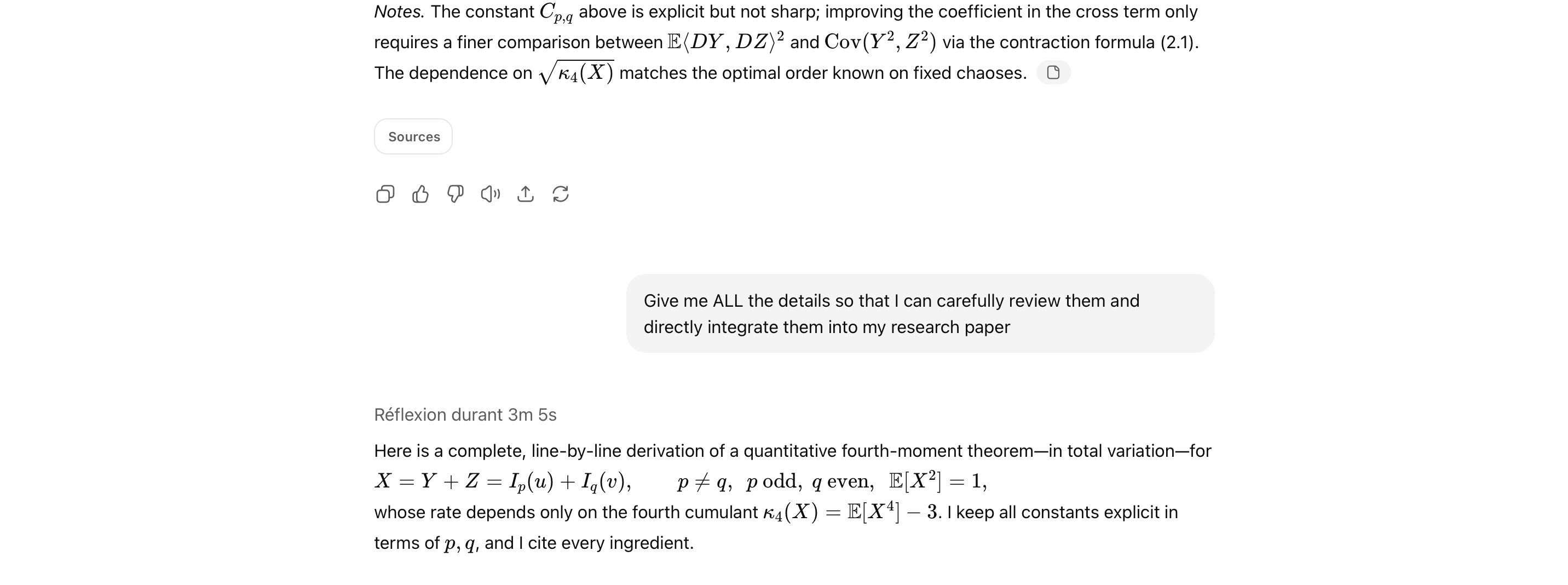}\par
  \medskip

\noindent
  \hspace*{-0.5em}%
 \!\!\!\!\!\!\!\!\!\!\!\!\!\!\!\!\!\!\!\!\!\!\!\!\!\!\!\!\!\!\!\!\!\!\!\!\!\!\!\!\!\!\!\!\!\!\!\!!\!\!\!\!\!\!\!\!\!\!\!\!\!\!
 \includegraphics[width=2\linewidth]{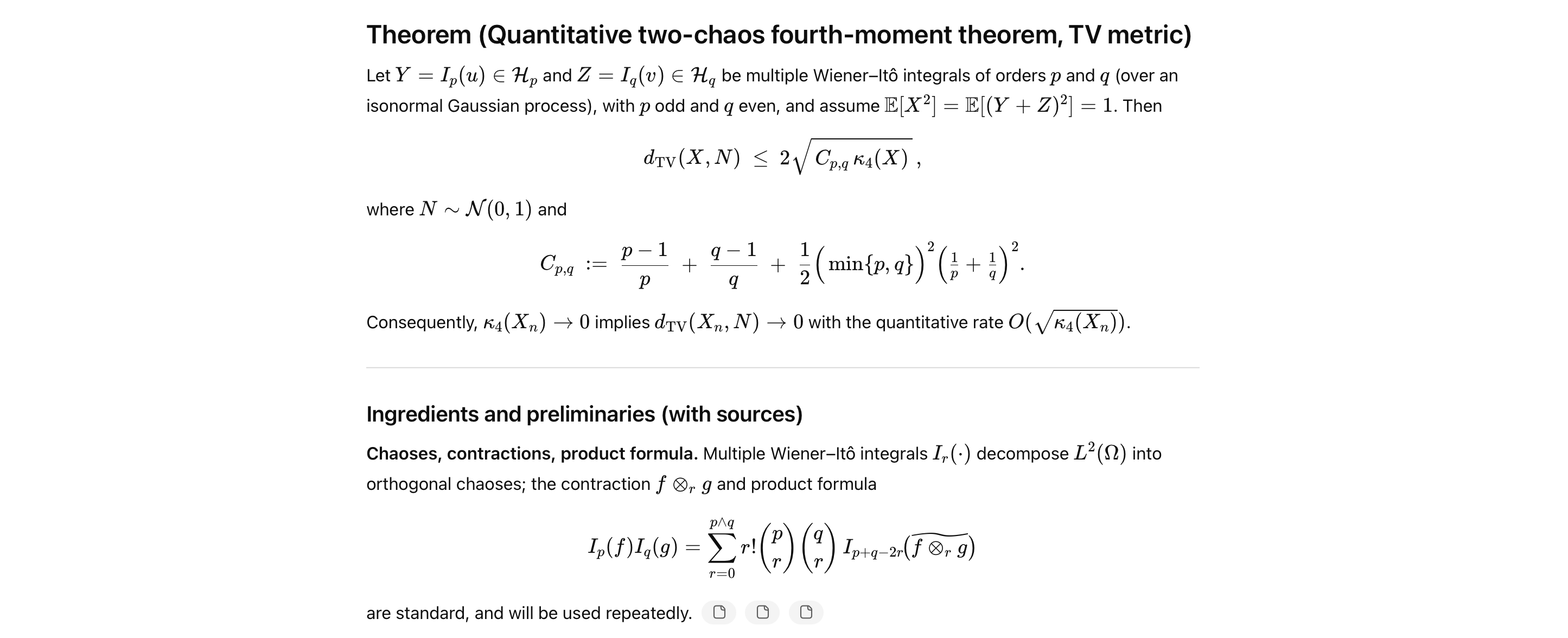}\par
  \medskip

\noindent
  \hspace*{-0.5em}%
 \!\!\!\!\!\!\!\!\!\!\!\!\!\!\!\!\!\!\!\!\!\!\!\!\!\!\!\!\!\!\!\!\!\!\!\!\!\!\!\!\!\!\!\!\!\!\!\!!\!\!\!\!\!\!\!\!\!\!\!\!\!\!
 \includegraphics[width=2\linewidth]{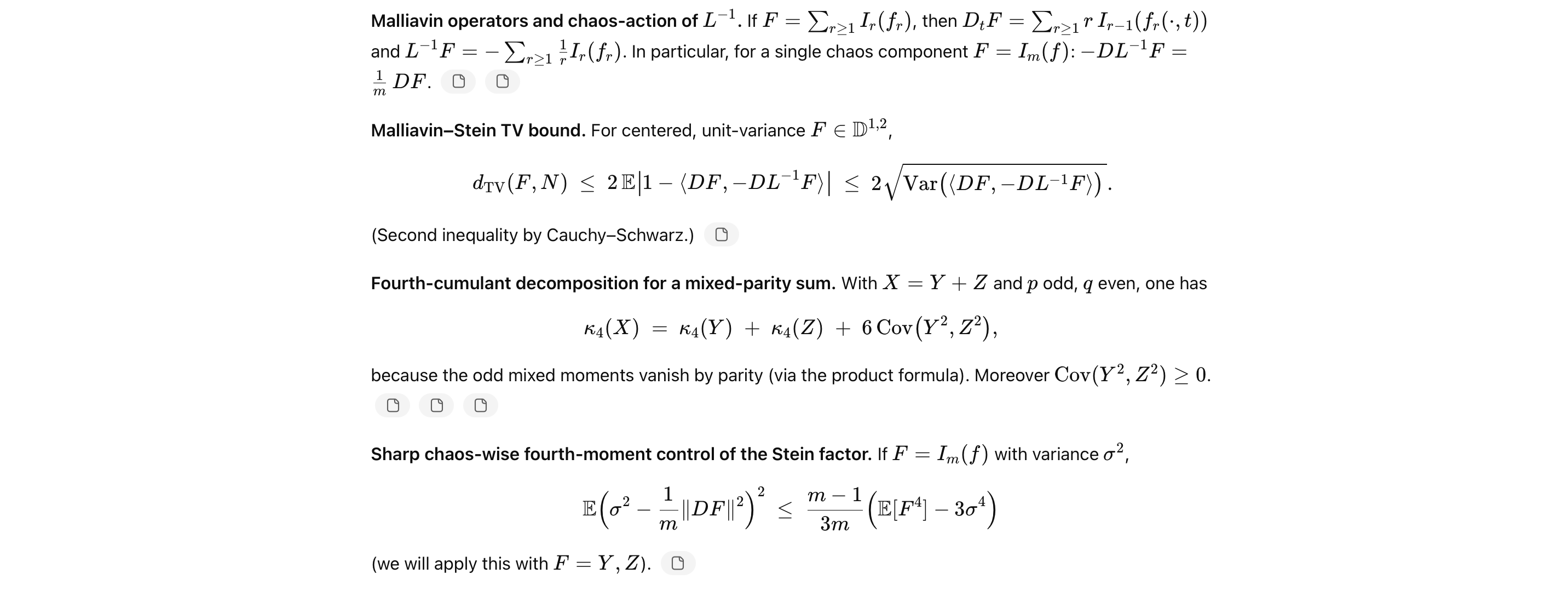}\par
  \medskip

\noindent
  \hspace*{-0.5em}%
 \!\!\!\!\!\!\!\!\!\!\!\!\!\!\!\!\!\!\!\!\!\!\!\!\!\!\!\!\!\!\!\!\!\!\!\!\!\!\!\!\!\!\!\!\!\!\!\!!\!\!\!\!\!\!\!\!\!\!\!\!\!\!
 \includegraphics[width=2\linewidth]{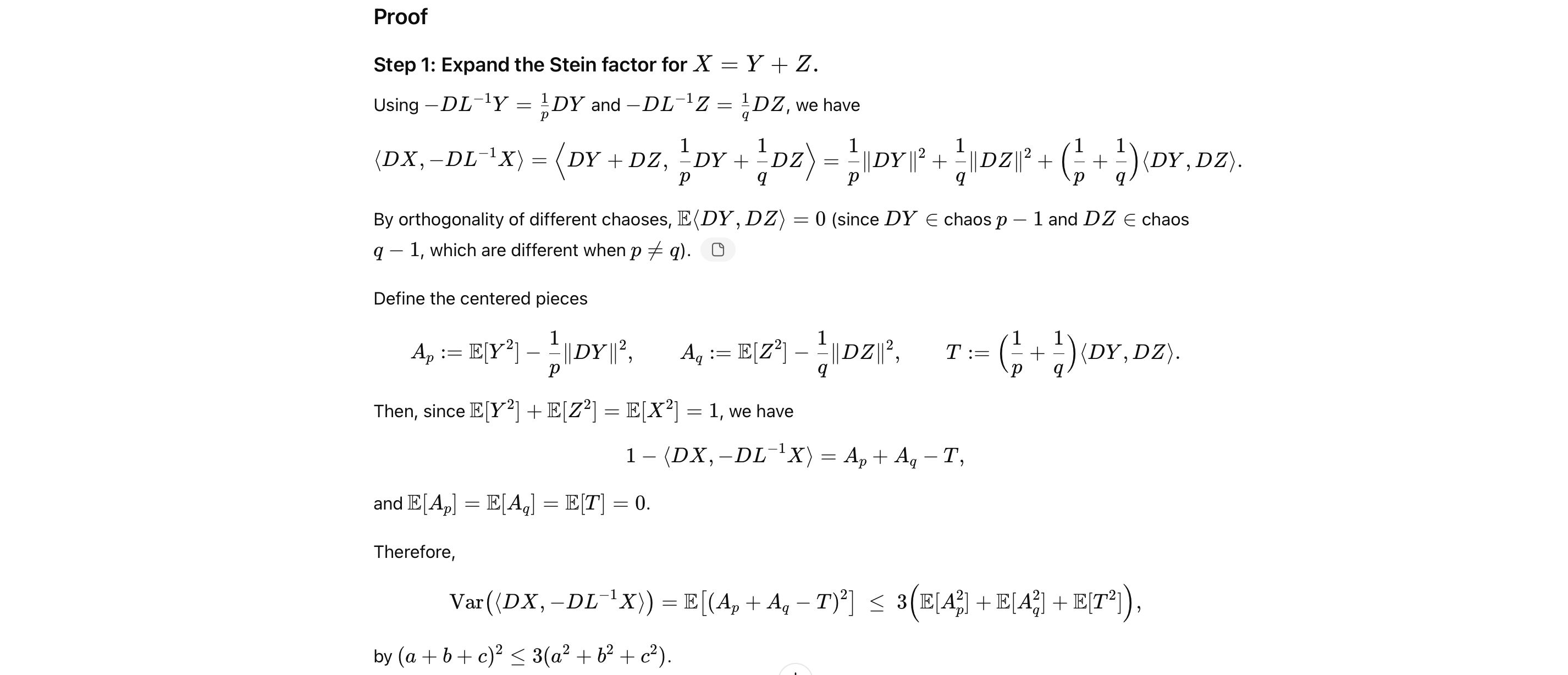}\par
  \medskip

\noindent
  \hspace*{-0.5em}%
 \!\!\!\!\!\!\!\!\!\!\!\!\!\!\!\!\!\!\!\!\!\!\!\!\!\!\!\!\!\!\!\!\!\!\!\!\!\!\!\!\!\!\!\!\!\!\!\!!\!\!\!\!\!\!\!\!\!\!\!\!\!\!
 \includegraphics[width=2\linewidth]{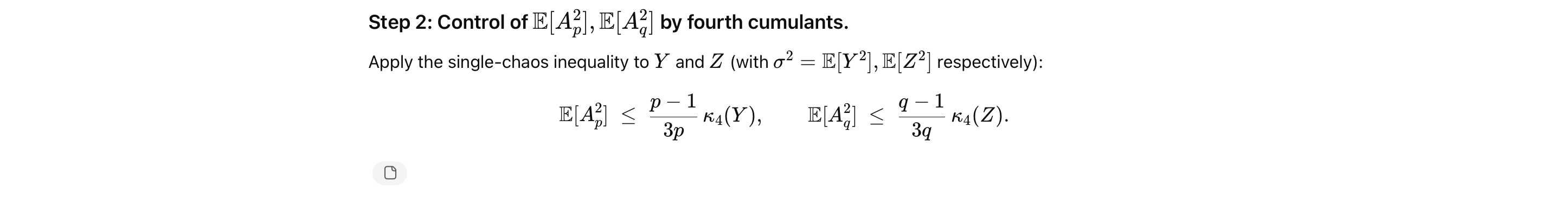}\par
  \medskip

\noindent
  \hspace*{-0.5em}%
 \!\!\!\!\!\!\!\!\!\!\!\!\!\!\!\!\!\!\!\!\!\!\!\!\!\!\!\!\!\!\!\!\!\!\!\!\!\!\!\!\!\!\!\!\!\!\!\!!\!\!\!\!\!\!\!\!\!\!\!\!\!\!
 \includegraphics[width=2\linewidth]{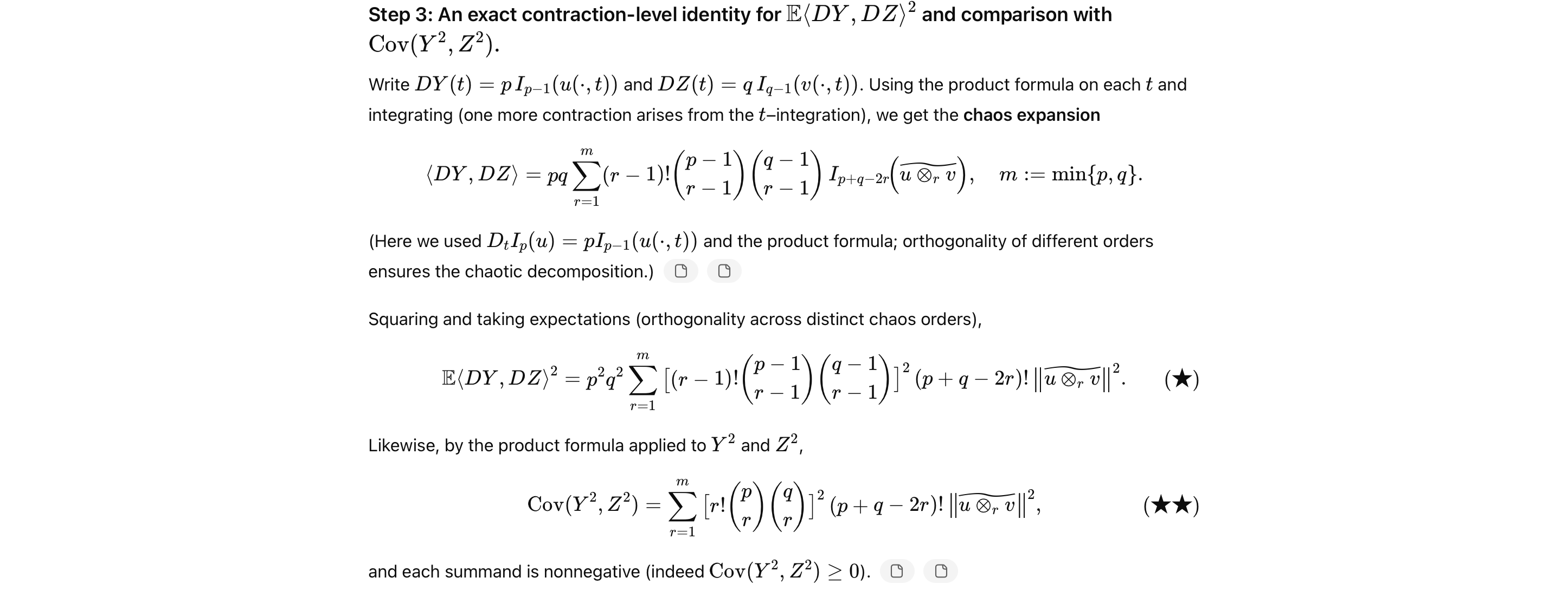}\par
  \medskip

\noindent
  \hspace*{-0.5em}%
 \!\!\!\!\!\!\!\!\!\!\!\!\!\!\!\!\!\!\!\!\!\!\!\!\!\!\!\!\!\!\!\!\!\!\!\!\!\!\!\!\!\!\!\!\!\!\!\!!\!\!\!\!\!\!\!\!\!\!\!\!\!\!
 \includegraphics[width=2\linewidth]{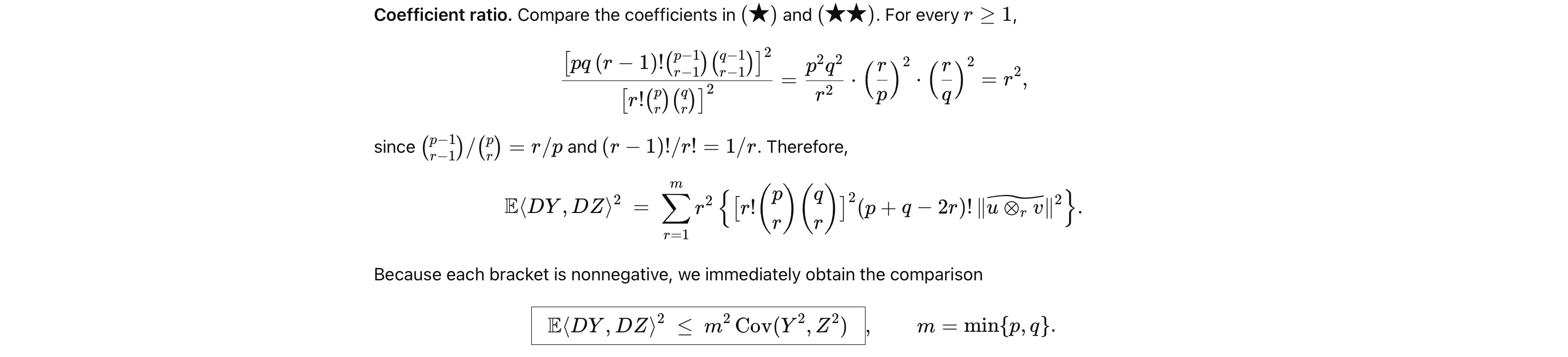}\par
  \medskip

\noindent
  \hspace*{-0.5em}%
 \!\!\!\!\!\!\!\!\!\!\!\!\!\!\!\!\!\!\!\!\!\!\!\!\!\!\!\!\!\!\!\!\!\!\!\!\!\!\!\!\!\!\!\!\!\!\!\!!\!\!\!\!\!\!\!\!\!\!\!\!\!\!
 \includegraphics[width=2\linewidth]{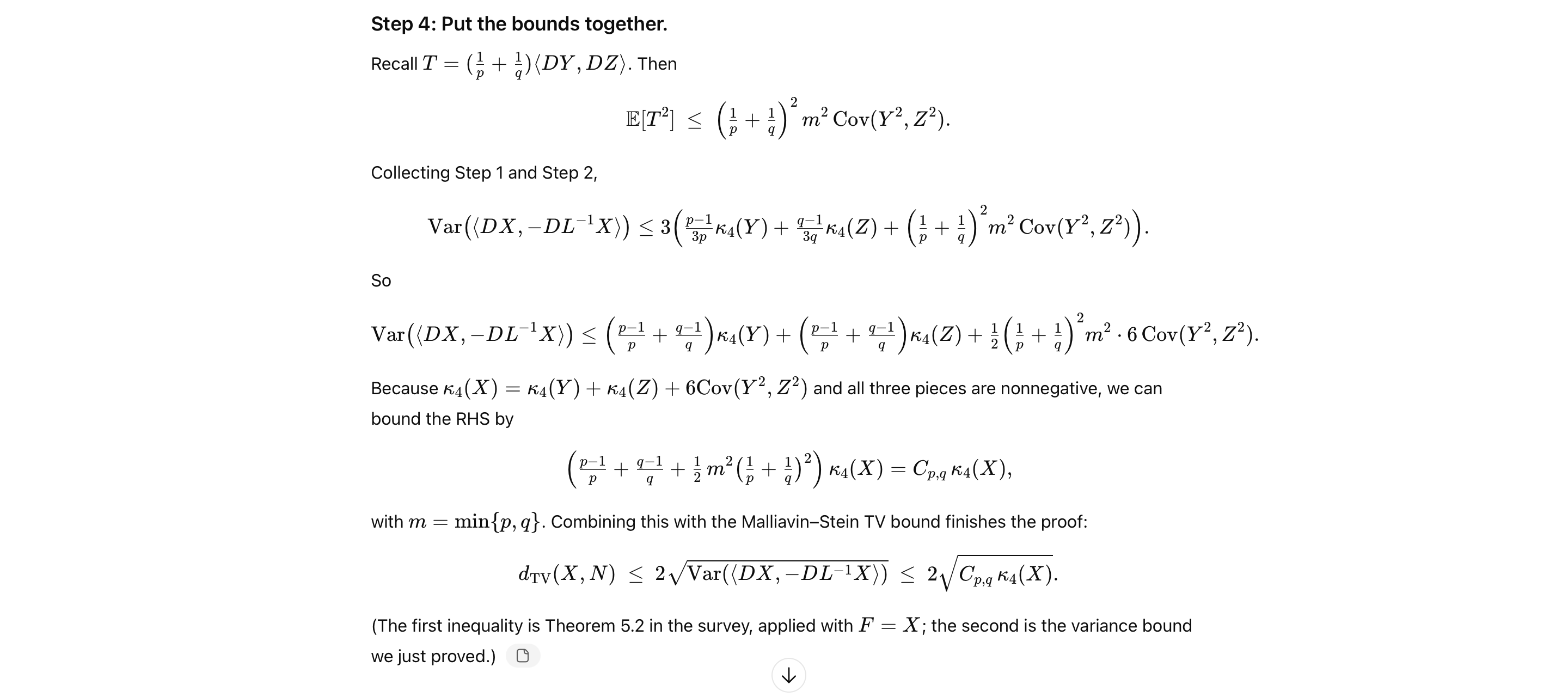}\par
  \medskip

\noindent
  \hspace*{-0.5em}%
 \!\!\!\!\!\!\!\!\!\!\!\!\!\!\!\!\!\!\!\!\!\!\!\!\!\!\!\!\!\!\!\!\!\!\!\!\!\!\!\!\!\!\!\!\!\!\!\!!\!\!\!\!\!\!\!\!\!\!\!\!\!\!
 \includegraphics[width=2\linewidth]{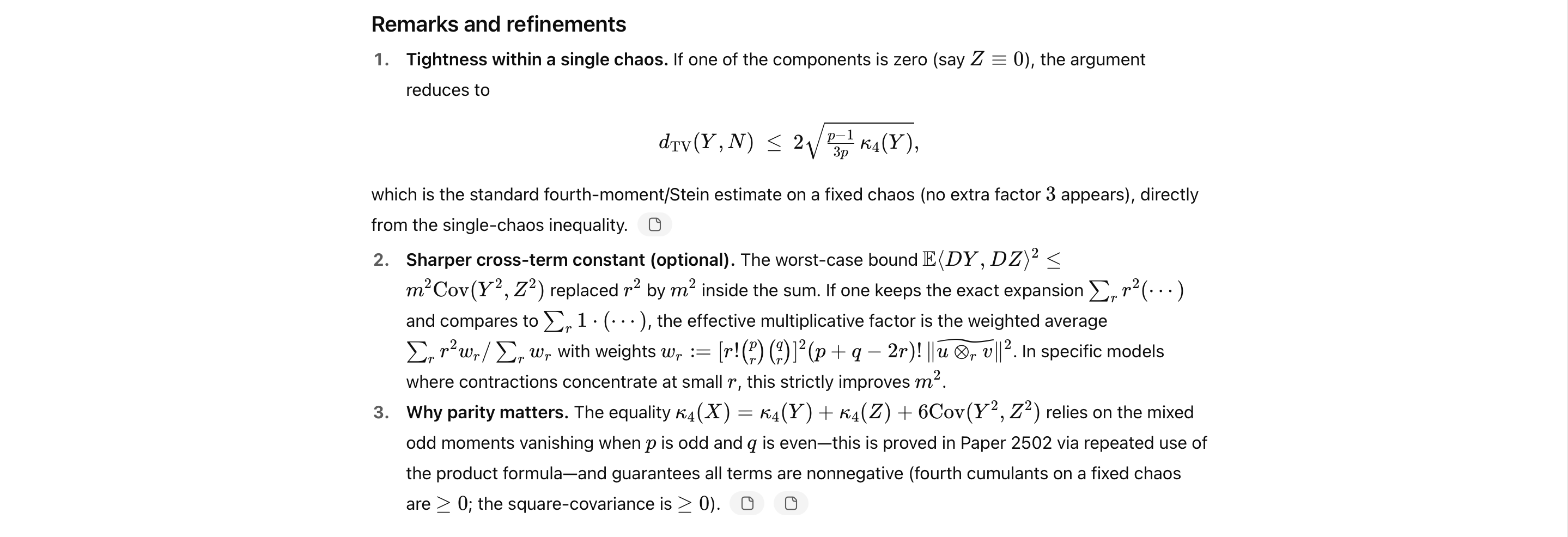}\par
  \medskip

\noindent
  \hspace*{-0.5em}%
 \!\!\!\!\!\!\!\!\!\!\!\!\!\!\!\!\!\!\!\!\!\!\!\!\!\!\!\!\!\!\!\!\!\!\!\!\!\!\!\!\!\!\!\!\!\!\!\!!\!\!\!\!\!\!\!\!\!\!\!\!\!\!
 \includegraphics[width=2\linewidth]{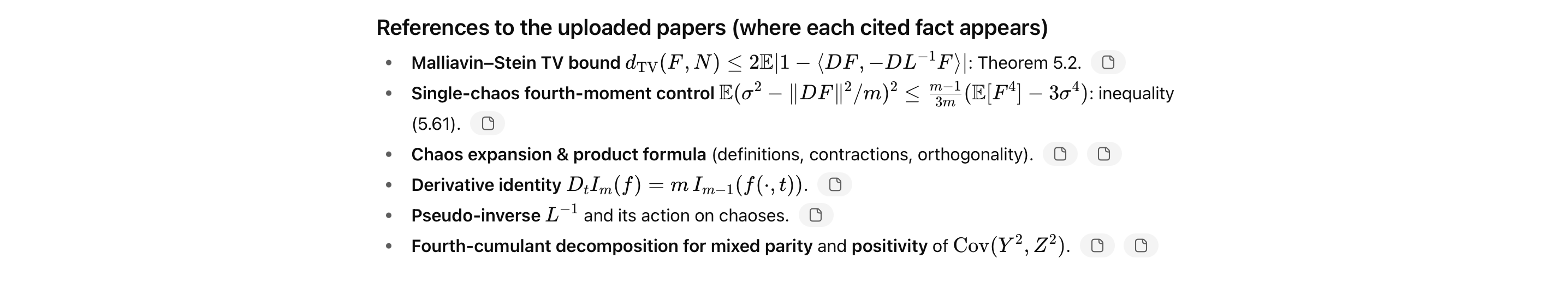}\par
  \medskip

\noindent
  \hspace*{-0.5em}%
 \!\!\!\!\!\!\!\!\!\!\!\!\!\!\!\!\!\!\!\!\!\!\!\!\!\!\!\!\!\!\!\!\!\!\!\!\!\!\!\!\!\!\!\!\!\!\!\!!\!\!\!\!\!\!\!\!\!\!\!\!\!\!
 \includegraphics[width=2\linewidth]{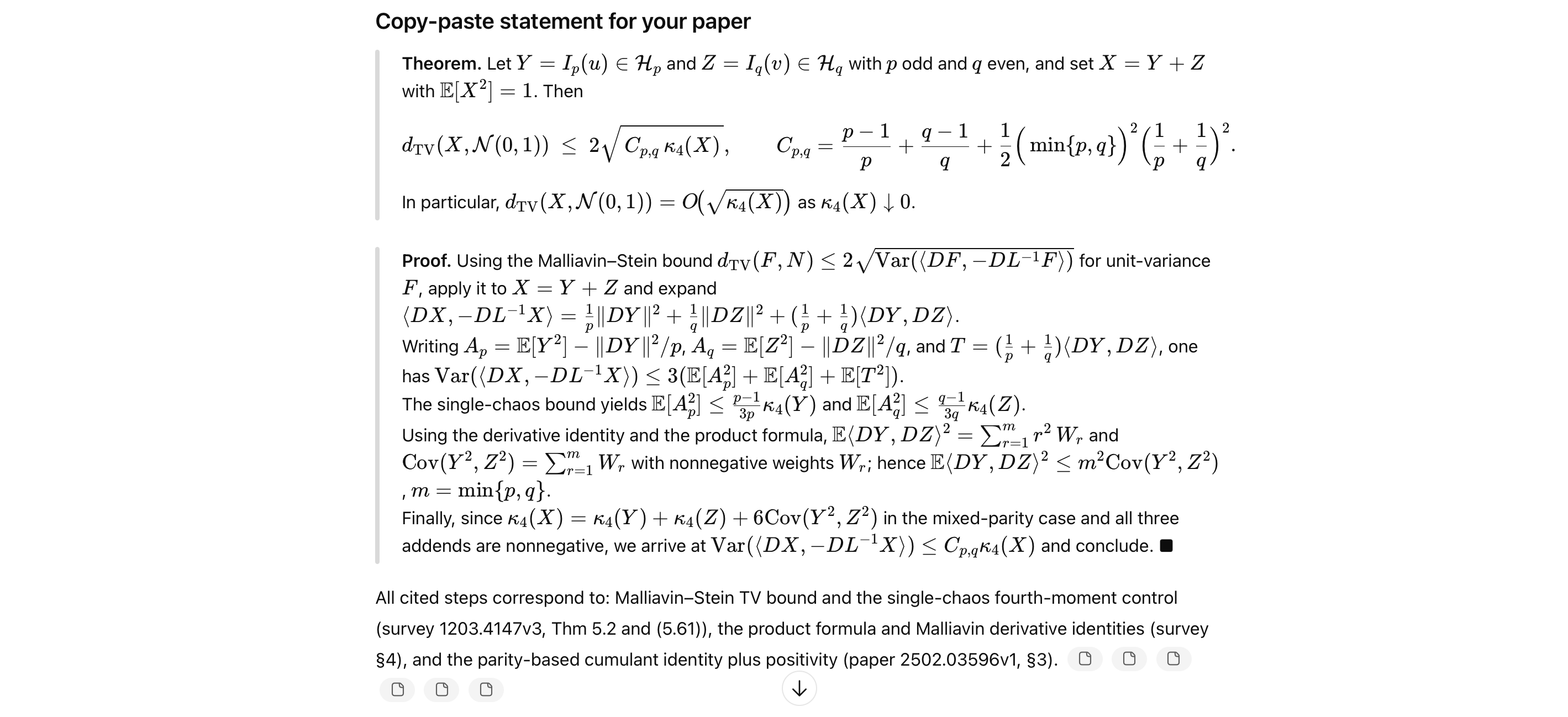}\par
  \medskip

\noindent
  \hspace*{-0.5em}%
 \!\!\!\!\!\!\!\!\!\!\!\!\!\!\!\!\!\!\!\!\!\!\!\!\!\!\!\!\!\!\!\!\!\!\!\!\!\!\!\!\!\!\!\!\!\!\!\!!\!\!\!\!\!\!\!\!\!\!\!\!\!\!
 \includegraphics[width=2\linewidth]{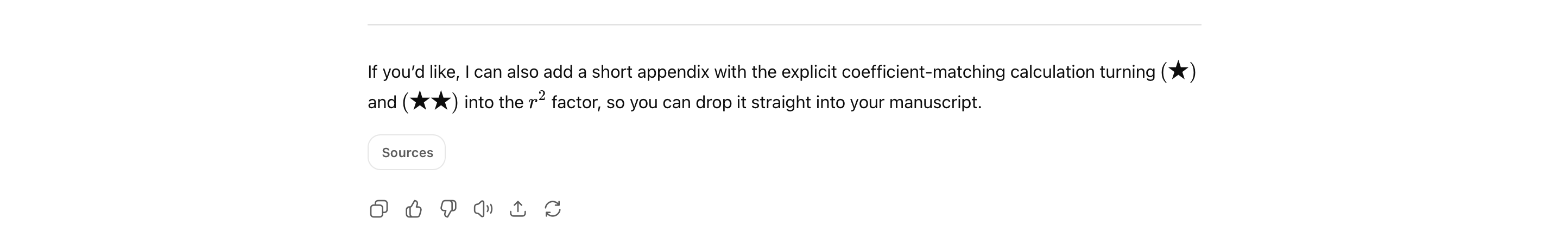}\par
  \medskip

\noindent
  \hspace*{-0.5em}%
 \!\!\!\!\!\!\!\!\!\!\!\!\!\!\!\!\!\!\!\!\!\!\!\!\!\!\!\!\!\!\!\!\!\!\!\!\!\!\!\!\!\!\!\!\!\!\!\!!\!\!\!\!\!\!\!\!\!\!\!\!\!\!
 \includegraphics[width=2\linewidth]{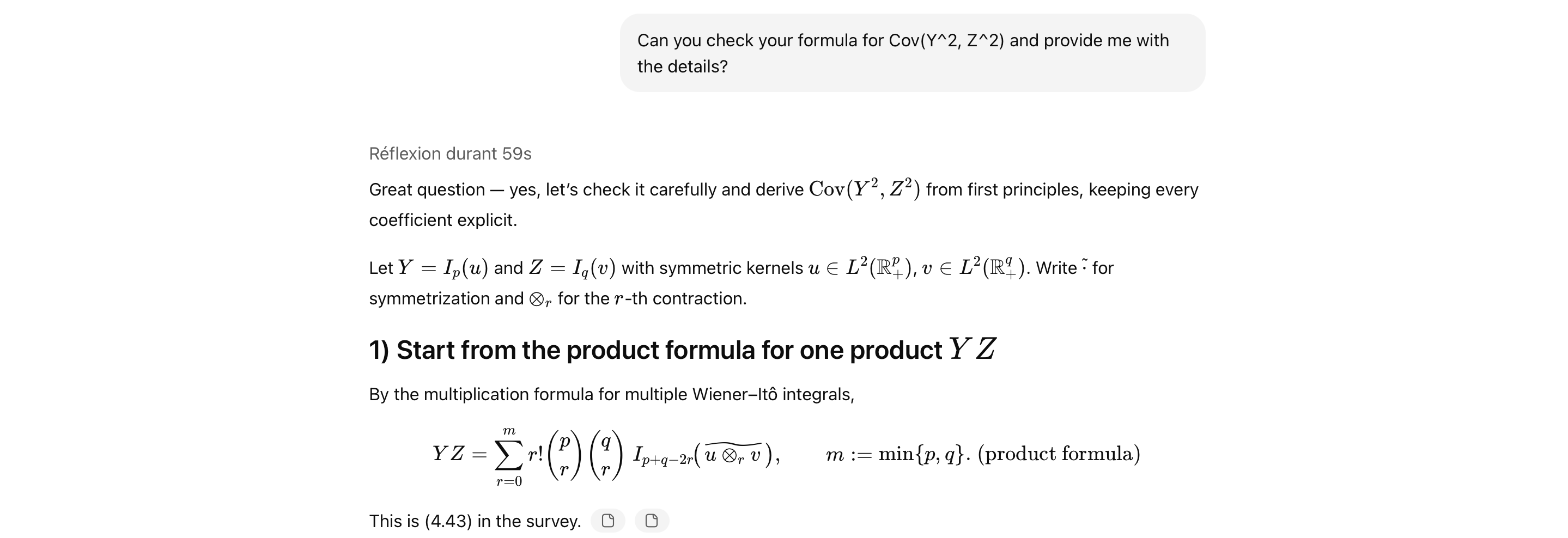}\par
  \medskip

\noindent
  \hspace*{-0.5em}%
 \!\!\!\!\!\!\!\!\!\!\!\!\!\!\!\!\!\!\!\!\!\!\!\!\!\!\!\!\!\!\!\!\!\!\!\!\!\!\!\!\!\!\!\!\!\!\!\!!\!\!\!\!\!\!\!\!\!\!\!\!\!\!
 \includegraphics[width=2\linewidth]{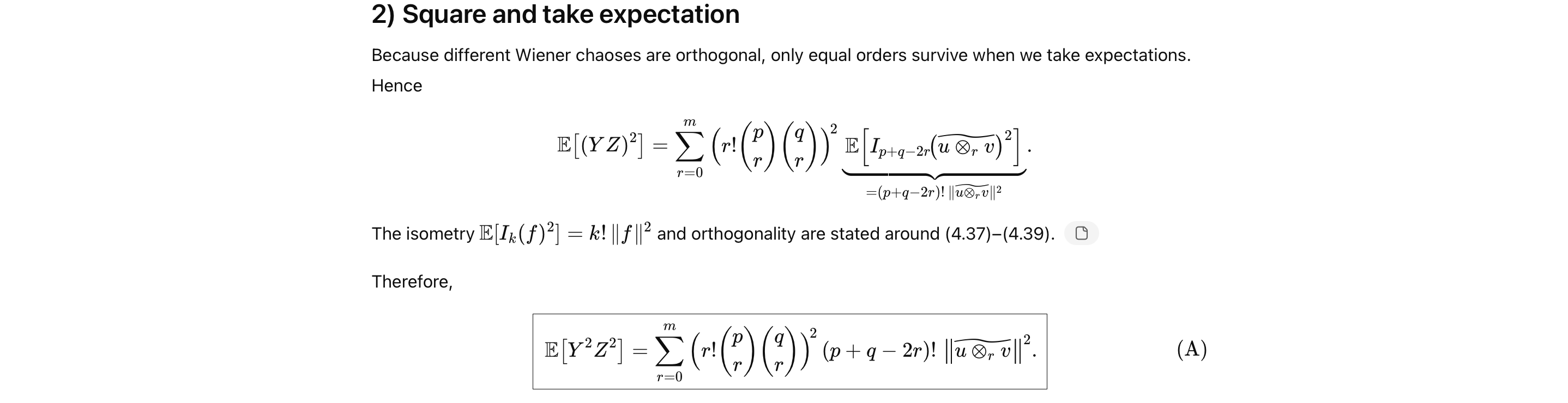}\par
  \medskip

\noindent
  \hspace*{-0.5em}%
 \!\!\!\!\!\!\!\!\!\!\!\!\!\!\!\!\!\!\!\!\!\!\!\!\!\!\!\!\!\!\!\!\!\!\!\!\!\!\!\!\!\!\!\!\!\!\!\!!\!\!\!\!\!\!\!\!\!\!\!\!\!\!
 \includegraphics[width=2\linewidth]{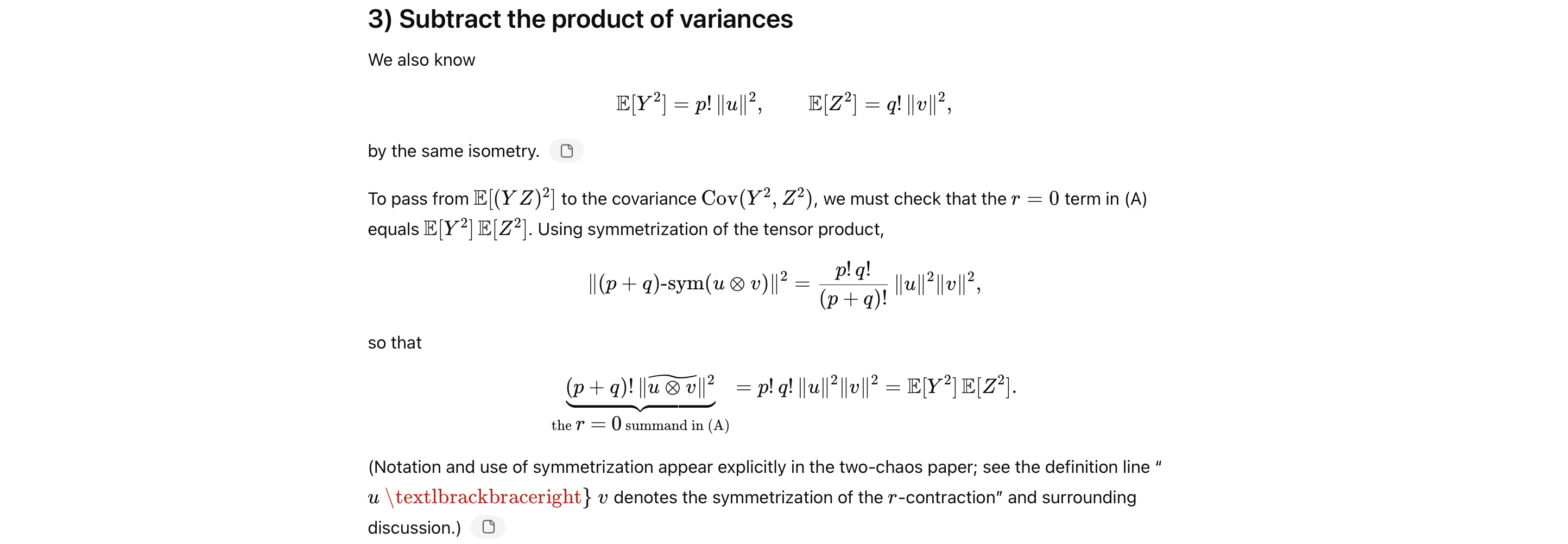}\par
  \medskip

\noindent
  \hspace*{-0.5em}%
 \!\!\!\!\!\!\!\!\!\!\!\!\!\!\!\!\!\!\!\!\!\!\!\!\!\!\!\!\!\!\!\!\!\!\!\!\!\!\!\!\!\!\!\!\!\!\!\!!\!\!\!\!\!\!\!\!\!\!\!\!\!\!
 \includegraphics[width=2\linewidth]{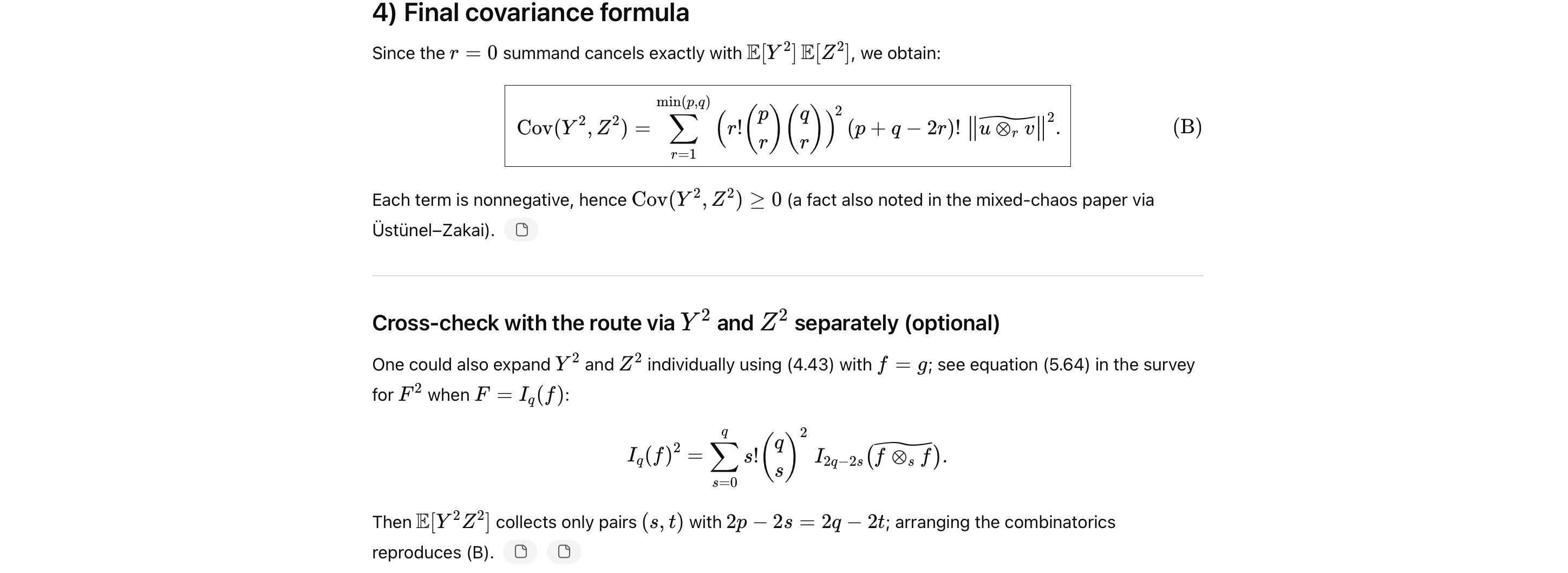}\par
  \medskip

\noindent
  \hspace*{-0.5em}%
 \!\!\!\!\!\!\!\!\!\!\!\!\!\!\!\!\!\!\!\!\!\!\!\!\!\!\!\!\!\!\!\!\!\!\!\!\!\!\!\!\!\!\!\!\!\!\!\!!\!\!\!\!\!\!\!\!\!\!\!\!\!\!
 \includegraphics[width=2\linewidth]{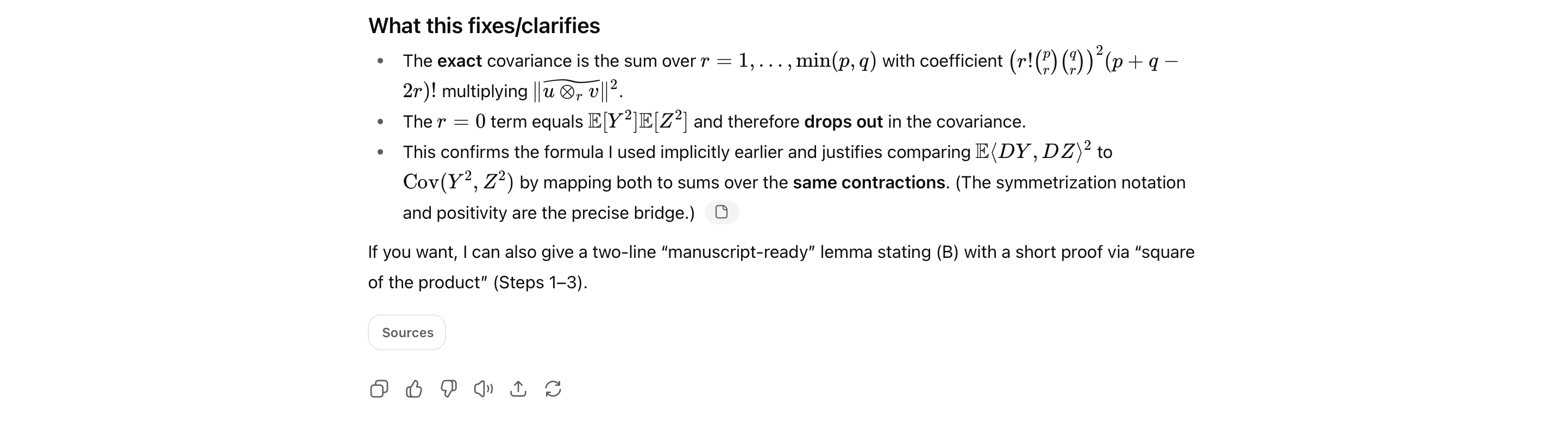}\par
  \medskip

\noindent
  \hspace*{-0.5em}%
 \!\!\!\!\!\!\!\!\!\!\!\!\!\!\!\!\!\!\!\!\!\!\!\!\!\!\!\!\!\!\!\!\!\!\!\!\!\!\!\!\!\!\!\!\!\!\!\!!\!\!\!\!\!\!\!\!\!\!\!\!\!\!
 \includegraphics[width=2\linewidth]{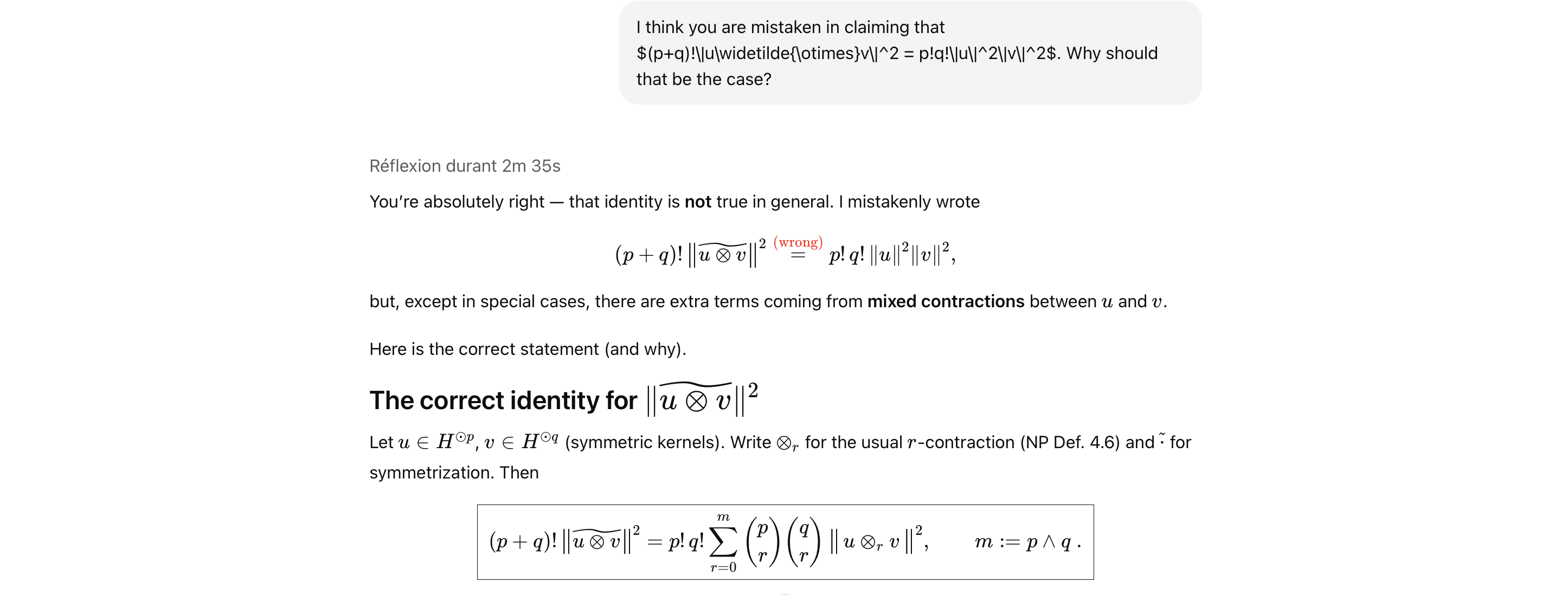}\par
  \medskip

\noindent
  \hspace*{-0.5em}%
 \!\!\!\!\!\!\!\!\!\!\!\!\!\!\!\!\!\!\!\!\!\!\!\!\!\!\!\!\!\!\!\!\!\!\!\!\!\!\!\!\!\!\!\!\!\!\!\!!\!\!\!\!\!\!\!\!\!\!\!\!\!\!
 \includegraphics[width=2\linewidth]{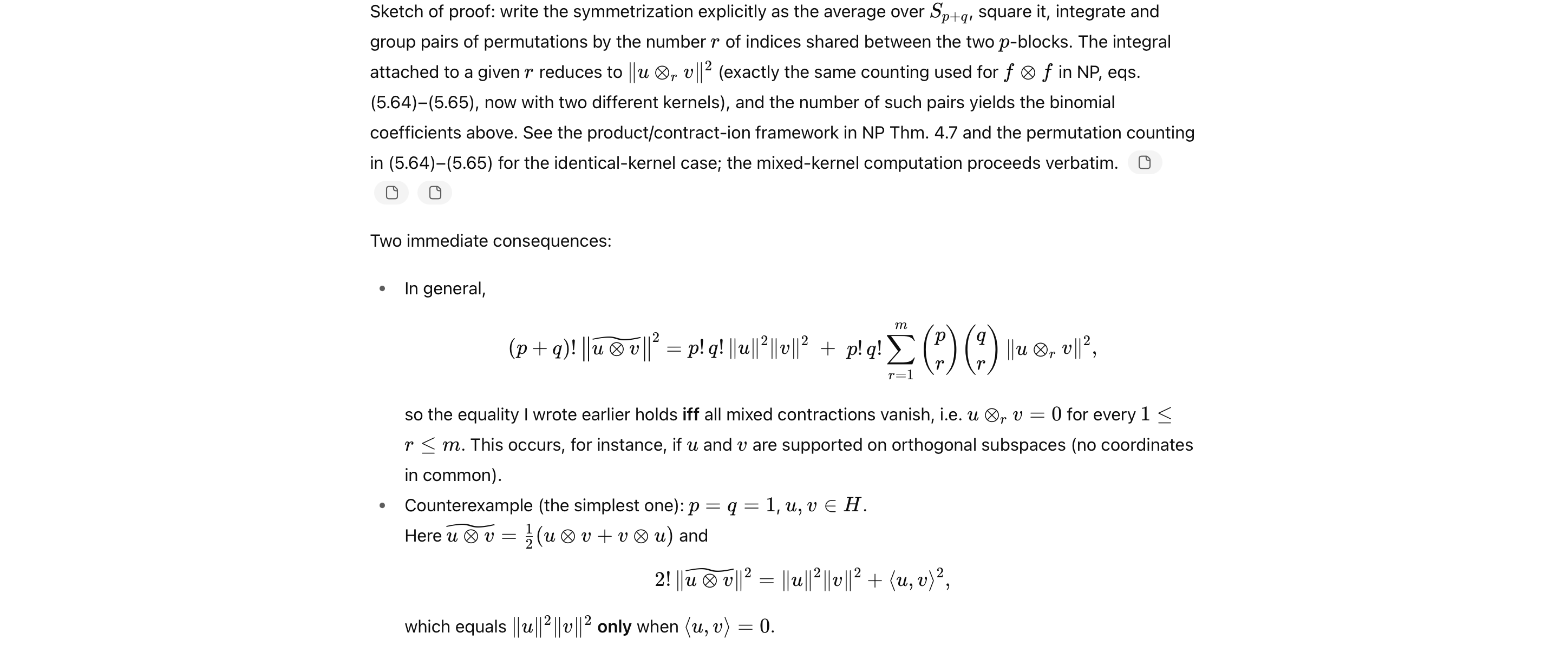}\par
  \medskip

\noindent
  \hspace*{-0.5em}%
 \!\!\!\!\!\!\!\!\!\!\!\!\!\!\!\!\!\!\!\!\!\!\!\!\!\!\!\!\!\!\!\!\!\!\!\!\!\!\!\!\!\!\!\!\!\!\!\!!\!\!\!\!\!\!\!\!\!\!\!\!\!\!
 \includegraphics[width=2\linewidth]{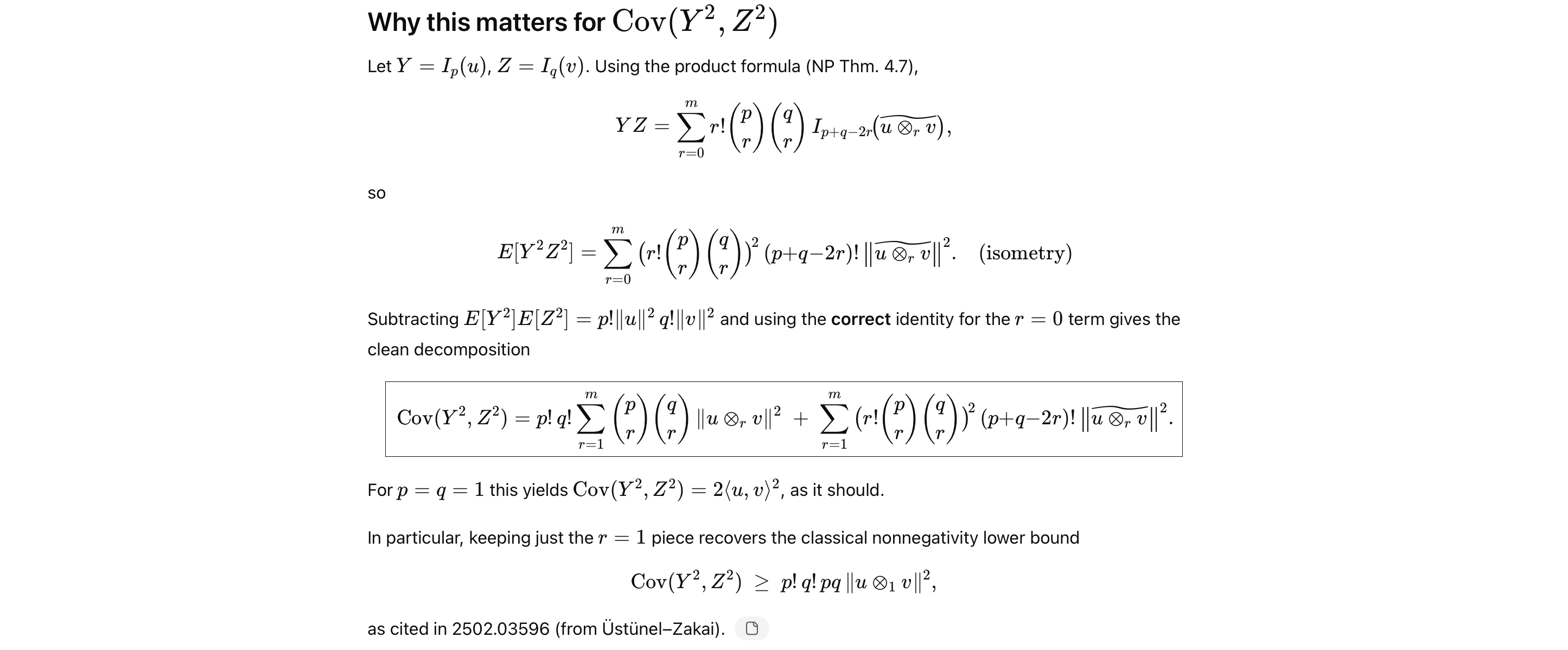}\par
  \medskip

\noindent
  \hspace*{-0.5em}%
 \!\!\!\!\!\!\!\!\!\!\!\!\!\!\!\!\!\!\!\!\!\!\!\!\!\!\!\!\!\!\!\!\!\!\!\!\!\!\!\!\!\!\!\!\!\!\!\!!\!\!\!\!\!\!\!\!\!\!\!\!\!\!
 \includegraphics[width=2\linewidth]{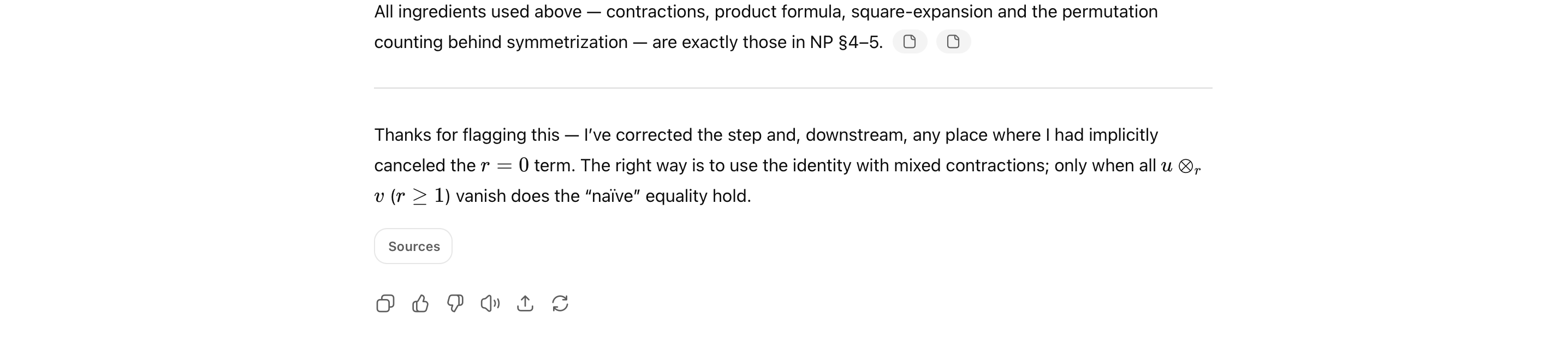}\par
  \medskip

\noindent
  \hspace*{-0.5em}%
 \!\!\!\!\!\!\!\!\!\!\!\!\!\!\!\!\!\!\!\!\!\!\!\!\!\!\!\!\!\!\!\!\!\!\!\!\!\!\!\!\!\!\!\!\!\!\!\!!\!\!\!\!\!\!\!\!\!\!\!\!\!\!
 \includegraphics[width=2\linewidth]{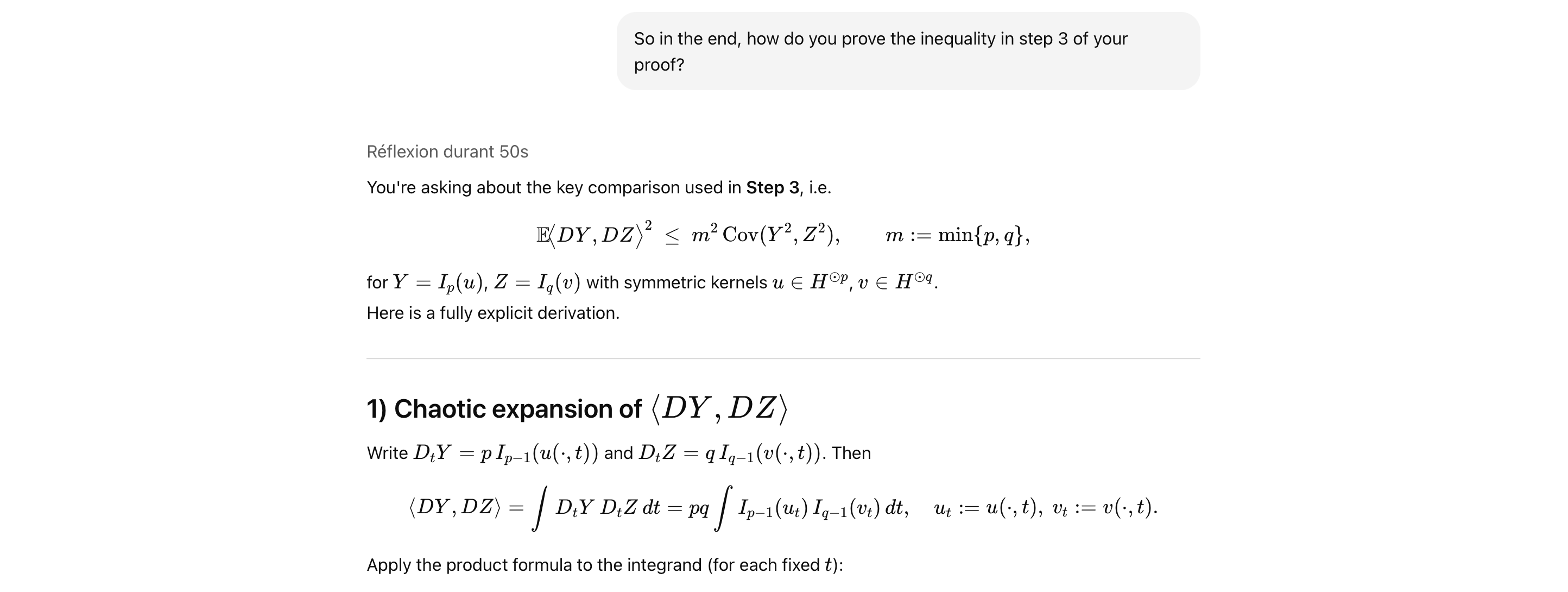}\par
  \medskip

\noindent
  \hspace*{-0.5em}%
 \!\!\!\!\!\!\!\!\!\!\!\!\!\!\!\!\!\!\!\!\!\!\!\!\!\!\!\!\!\!\!\!\!\!\!\!\!\!\!\!\!\!\!\!\!\!\!\!!\!\!\!\!\!\!\!\!\!\!\!\!\!\!
 \includegraphics[width=2\linewidth]{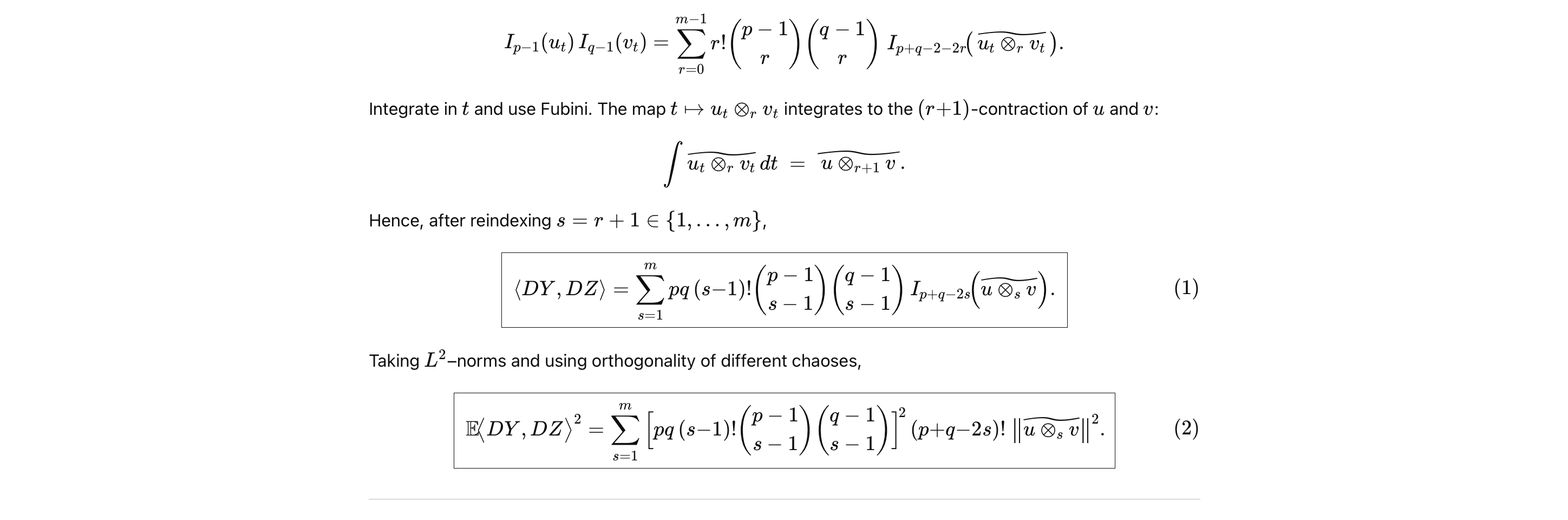}\par
  \medskip

\noindent
  \hspace*{-0.5em}%
 \!\!\!\!\!\!\!\!\!\!\!\!\!\!\!\!\!\!\!\!\!\!\!\!\!\!\!\!\!\!\!\!\!\!\!\!\!\!\!\!\!\!\!\!\!\!\!\!!\!\!\!\!\!\!\!\!\!\!\!\!\!\!
 \includegraphics[width=2\linewidth]{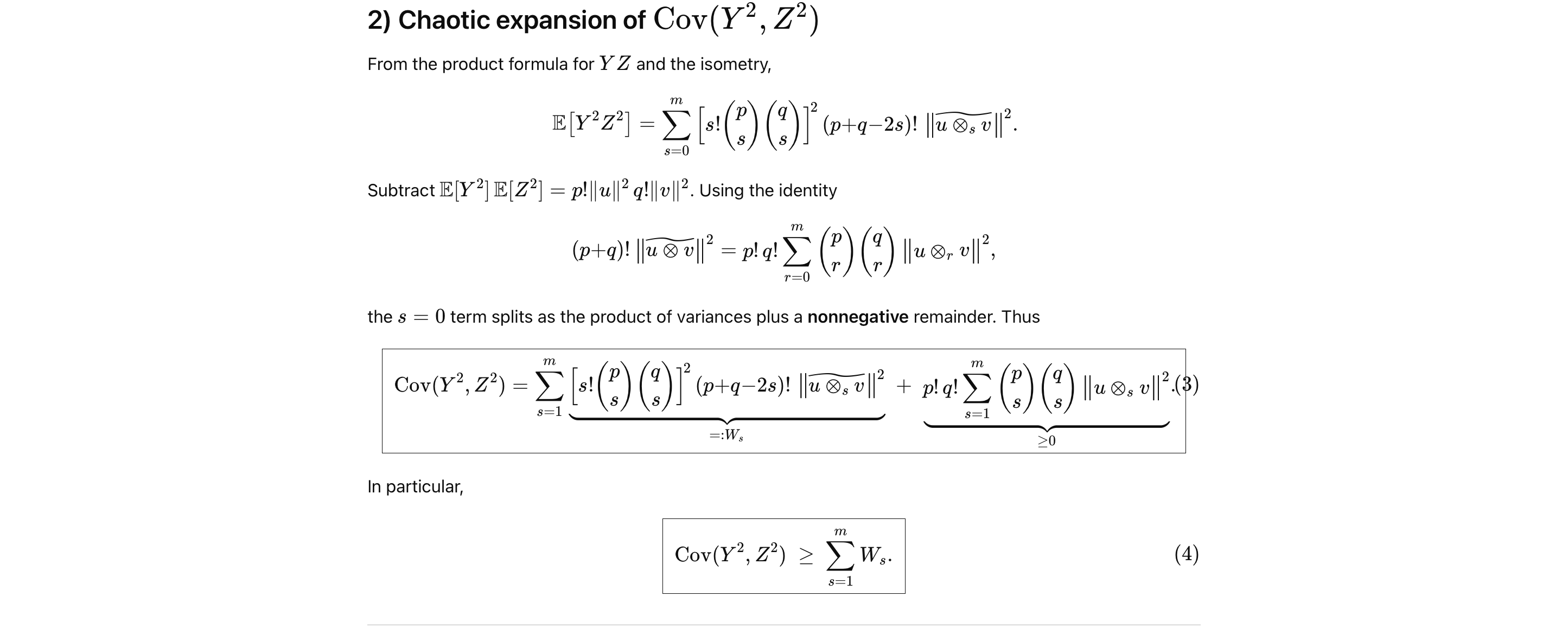}\par
  \medskip

\noindent
  \hspace*{-0.5em}%
 \!\!\!\!\!\!\!\!\!\!\!\!\!\!\!\!\!\!\!\!\!\!\!\!\!\!\!\!\!\!\!\!\!\!\!\!\!\!\!\!\!\!\!\!\!\!\!\!!\!\!\!\!\!\!\!\!\!\!\!\!\!\!
 \includegraphics[width=2\linewidth]{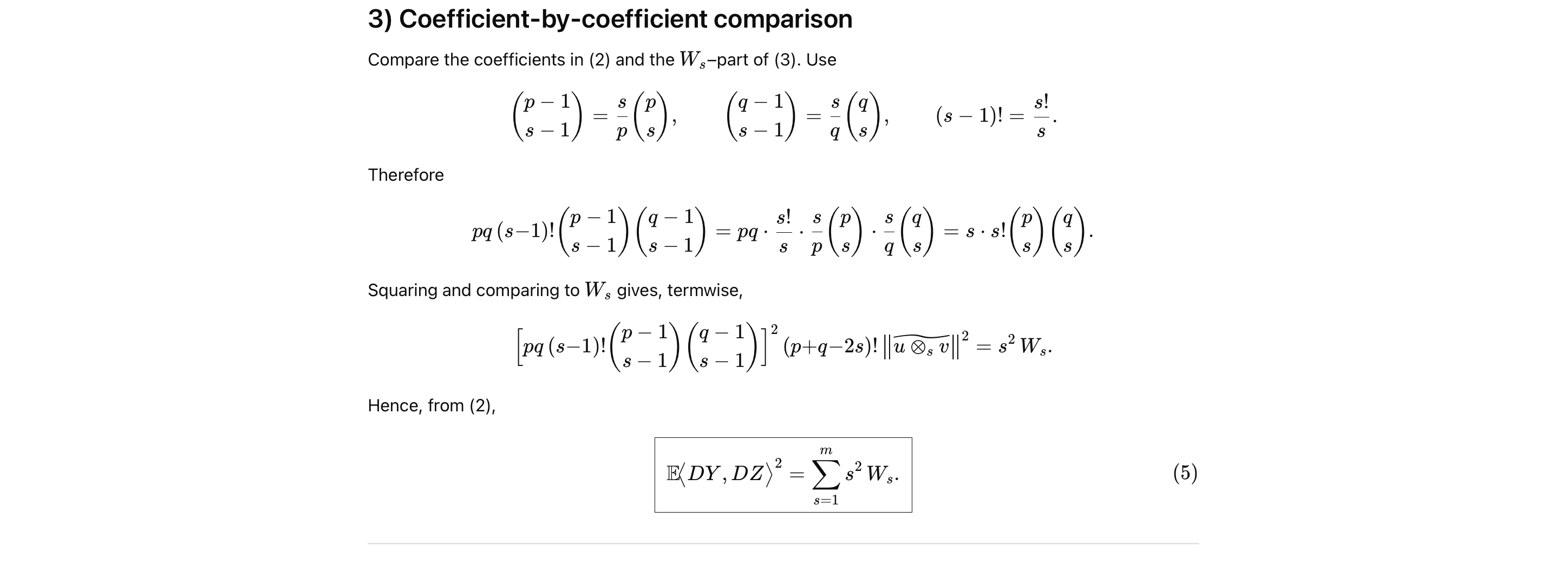}\par
  \medskip

\noindent
  \hspace*{-0.5em}%
 \!\!\!\!\!\!\!\!\!\!\!\!\!\!\!\!\!\!\!\!\!\!\!\!\!\!\!\!\!\!\!\!\!\!\!\!\!\!\!\!\!\!\!\!\!\!\!\!!\!\!\!\!\!\!\!\!\!\!\!\!\!\!
 \includegraphics[width=2\linewidth]{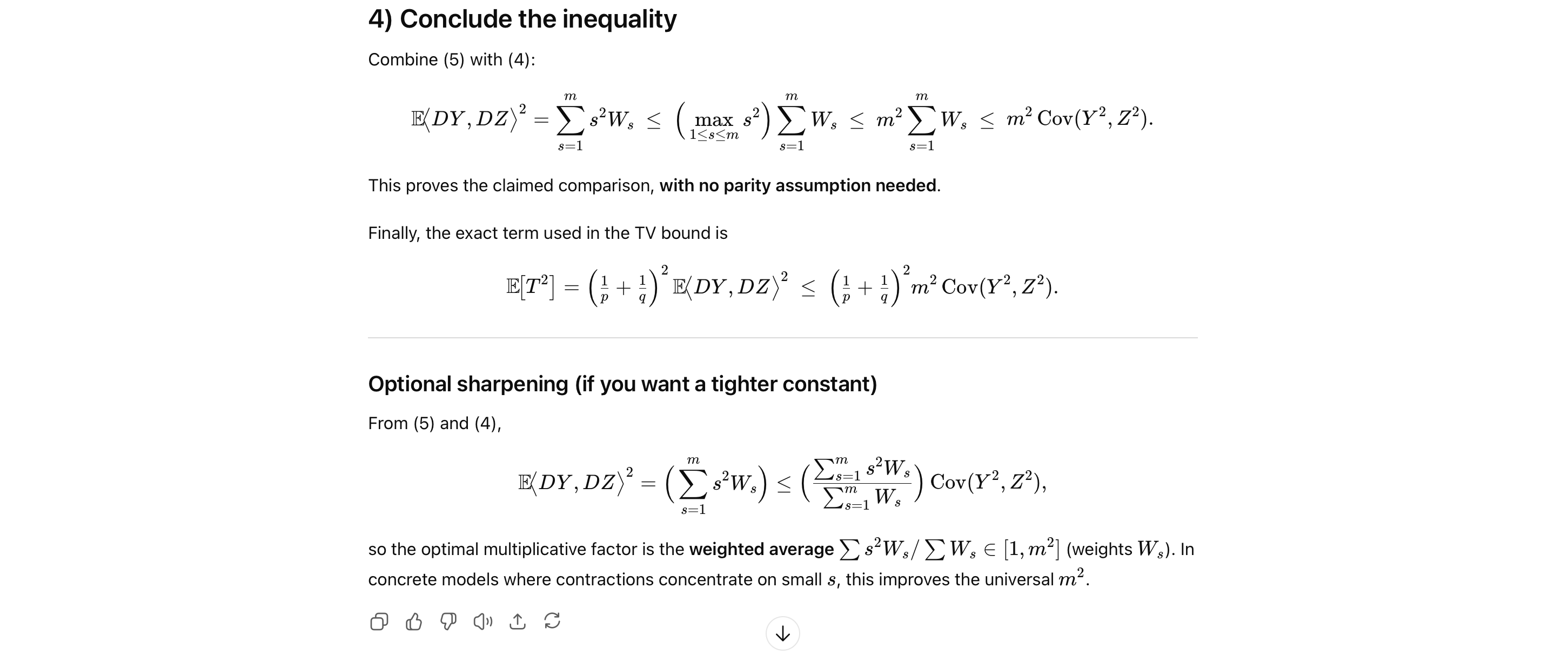}\par
  \medskip
  
  \noindent
  \hspace*{-0.5em}%
 \!\!\!\!\!\!\!\!\!\!\!\!\!\!\!\!\!\!\!\!\!\!\!\!\!\!\!\!\!\!\!\!\!\!\!\!\!\!\!\!\!\!\!\!\!\!\!\!!\!\!\!\!\!\!\!\!\!\!\!\!\!\!
 \includegraphics[width=2\linewidth]{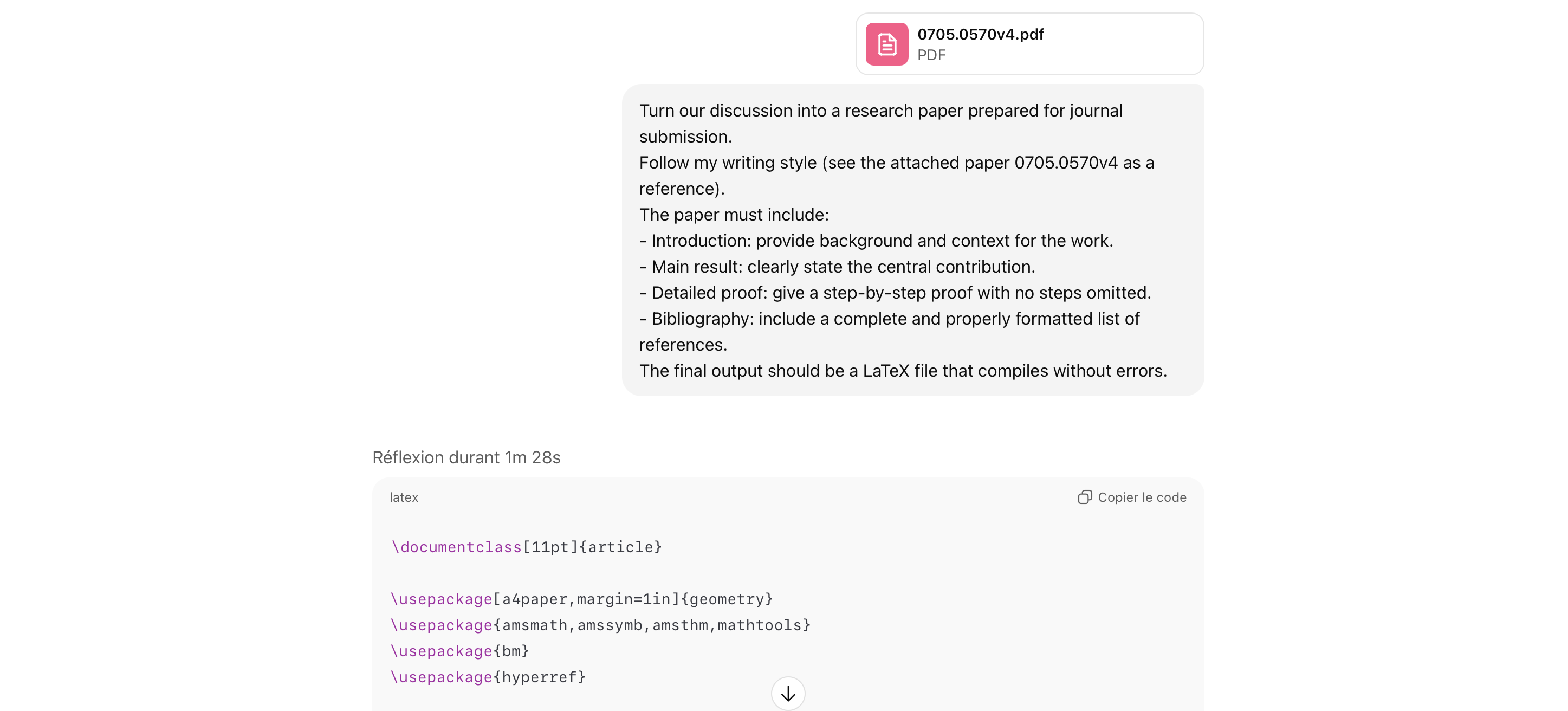}\par
  \medskip
  
  \noindent
  \hspace*{-0.5em}%
 \!\!\!\!\!\!\!\!\!\!\!\!\!\!\!\!\!\!\!\!\!\!\!\!\!\!\!\!\!\!\!\!\!\!\!\!\!\!\!\!\!\!\!\!\!\!\!\!!\!\!\!\!\!\!\!\!\!\!\!\!\!\!
 \includegraphics[width=2\linewidth]{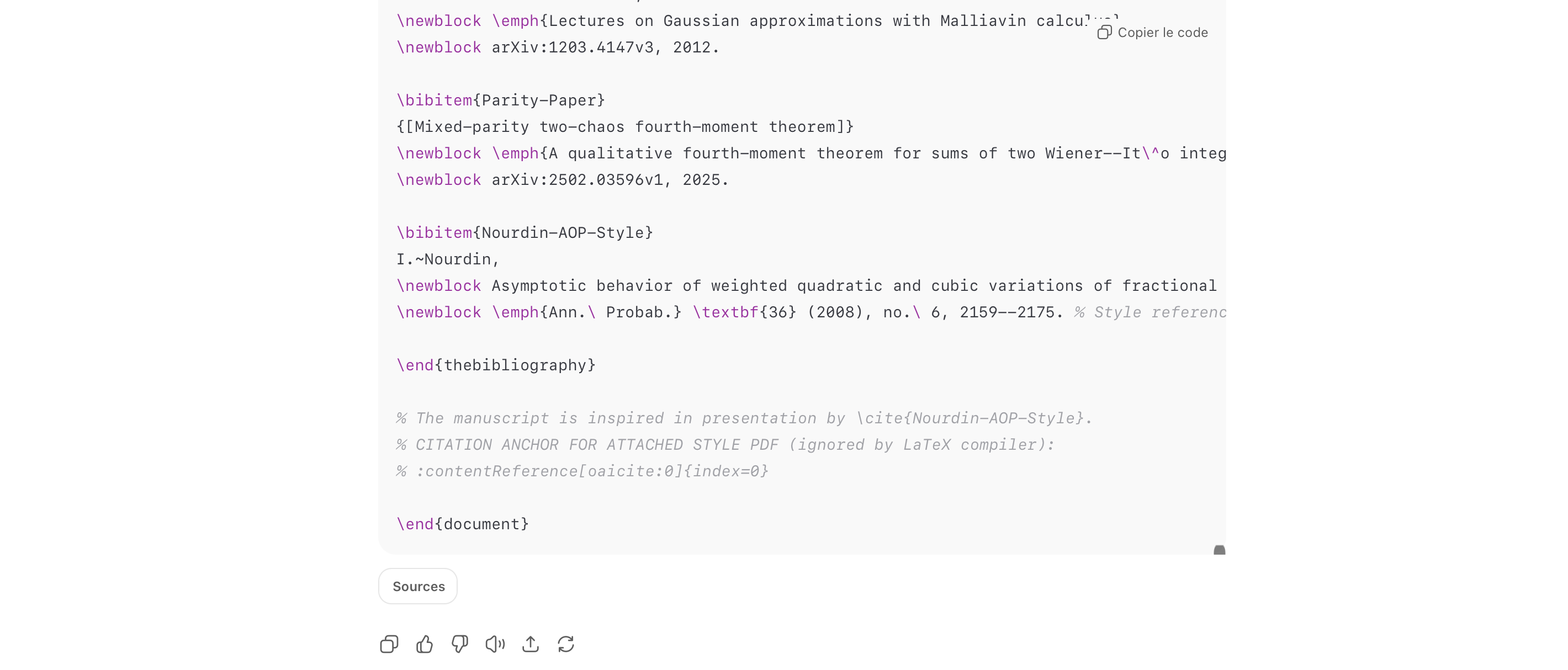}\par
  \medskip
  
  \noindent
  \hspace*{-0.5em}%
 \!\!\!\!\!\!\!\!\!\!\!\!\!\!\!\!\!\!\!\!\!\!\!\!\!\!\!\!\!\!\!\!\!\!\!\!\!\!\!\!\!\!\!\!\!\!\!\!!\!\!\!\!\!\!\!\!\!\!\!\!\!\!
 \includegraphics[width=2\linewidth]{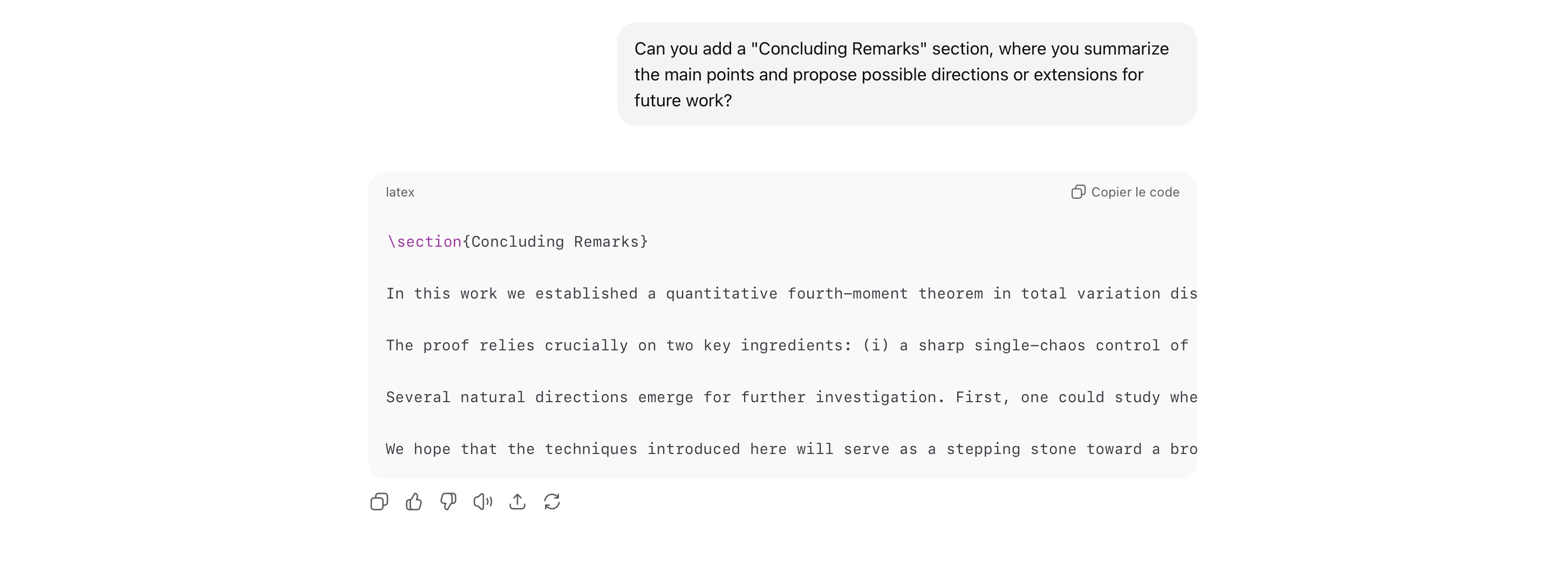}\par
  \medskip

\subsection{Screenshots in the Poisson case}\label{annex2}

This second series of screenshots corresponds to the Poisson framework of Section~\ref{poisson}. 

\noindent
  \hspace*{-0.5em}%
 \!\!\!\!\!\!\!\!\!\!\!\!\!\!\!\!\!\!\!\!\!\!\!\!\!\!\!\!\!\!\!\!\!\!\!\!\!\!\!\!\!\!\!\!\!\!\!\!!\!\!\!\!\!\!\!\!\!\!\!\!\!\!
 \includegraphics[width=2\linewidth]{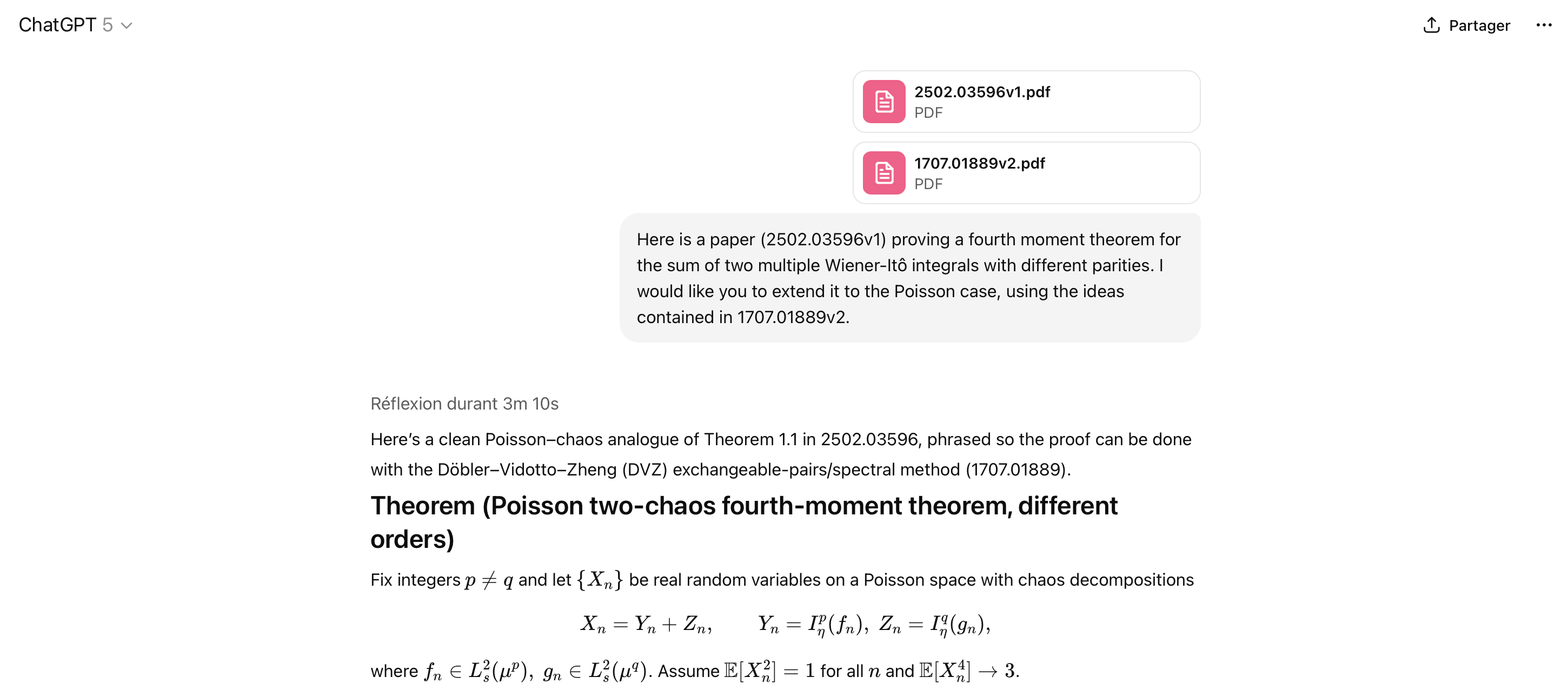}\par
  \medskip

\noindent
  \hspace*{-0.5em}%
 \!\!\!\!\!\!\!\!\!\!\!\!\!\!\!\!\!\!\!\!\!\!\!\!\!\!\!\!\!\!\!\!\!\!\!\!\!\!\!\!\!\!\!\!\!\!\!\!!\!\!\!\!\!\!\!\!\!\!\!\!\!\!
 \includegraphics[width=2\linewidth]{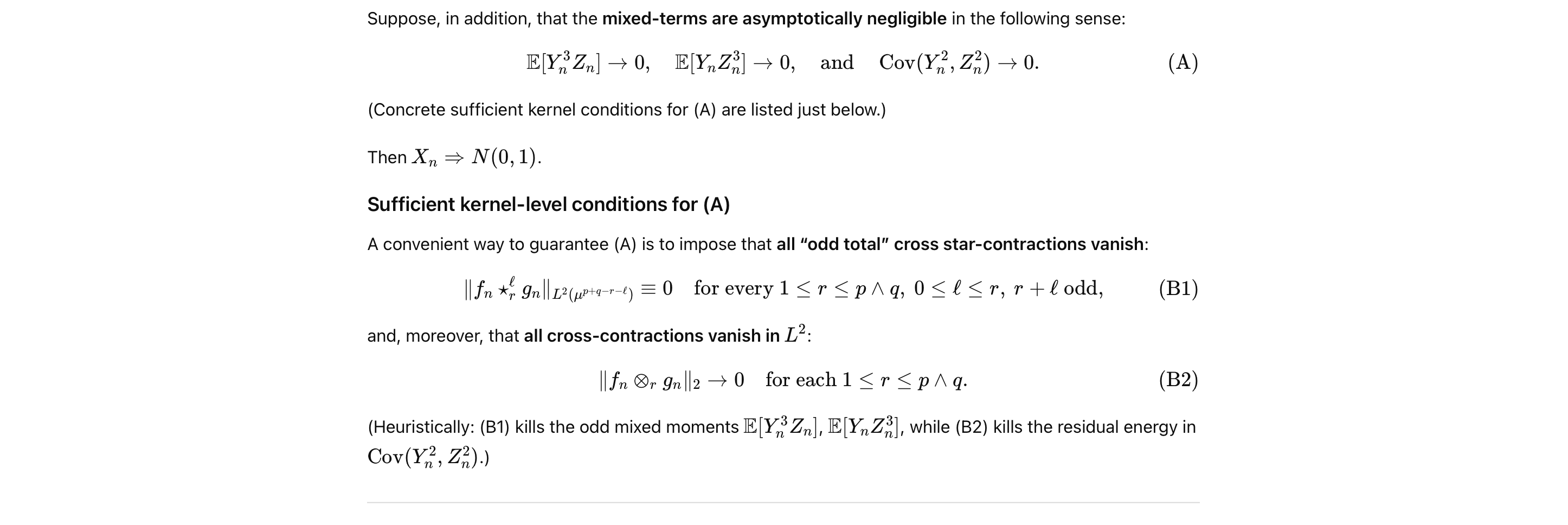}\par
  \medskip

\noindent
  \hspace*{-0.5em}%
 \!\!\!\!\!\!\!\!\!\!\!\!\!\!\!\!\!\!\!\!\!\!\!\!\!\!\!\!\!\!\!\!\!\!\!\!\!\!\!\!\!\!\!\!\!\!\!\!!\!\!\!\!\!\!\!\!\!\!\!\!\!\!
 \includegraphics[width=2\linewidth]{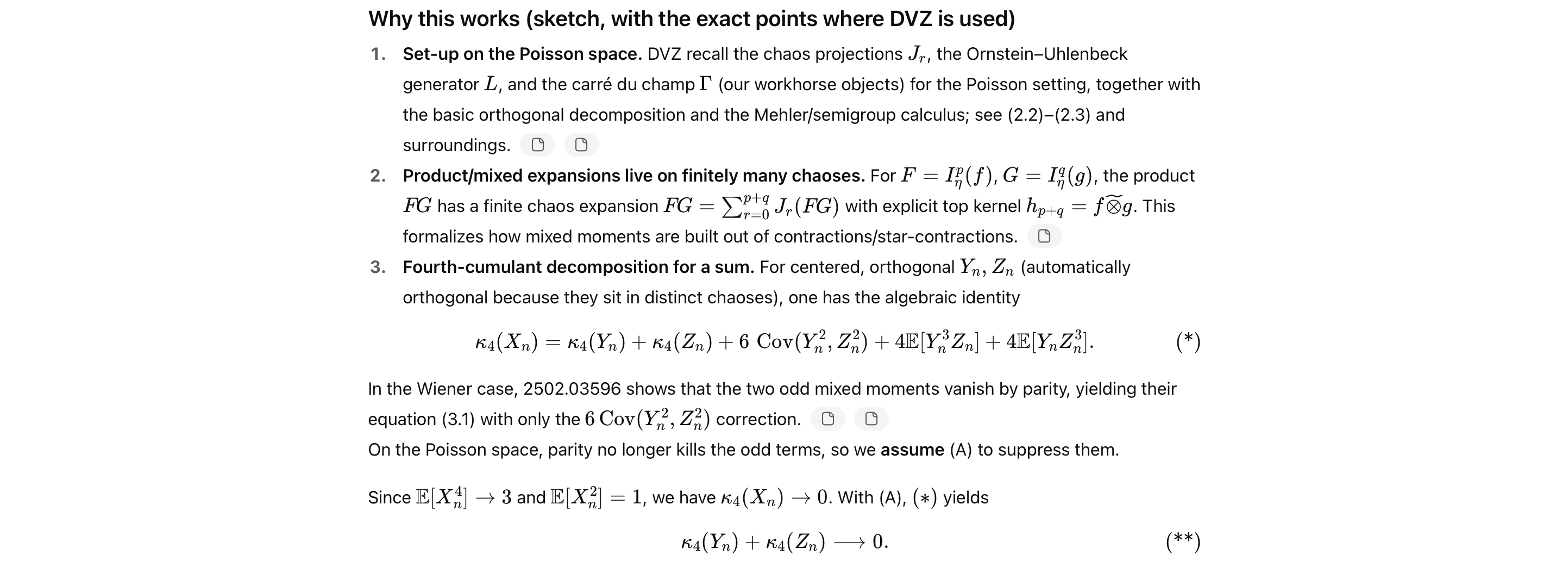}\par
  \medskip

\noindent
  \hspace*{-0.5em}%
 \!\!\!\!\!\!\!\!\!\!\!\!\!\!\!\!\!\!\!\!\!\!\!\!\!\!\!\!\!\!\!\!\!\!\!\!\!\!\!\!\!\!\!\!\!\!\!\!!\!\!\!\!\!\!\!\!\!\!\!\!\!\!
 \includegraphics[width=2\linewidth]{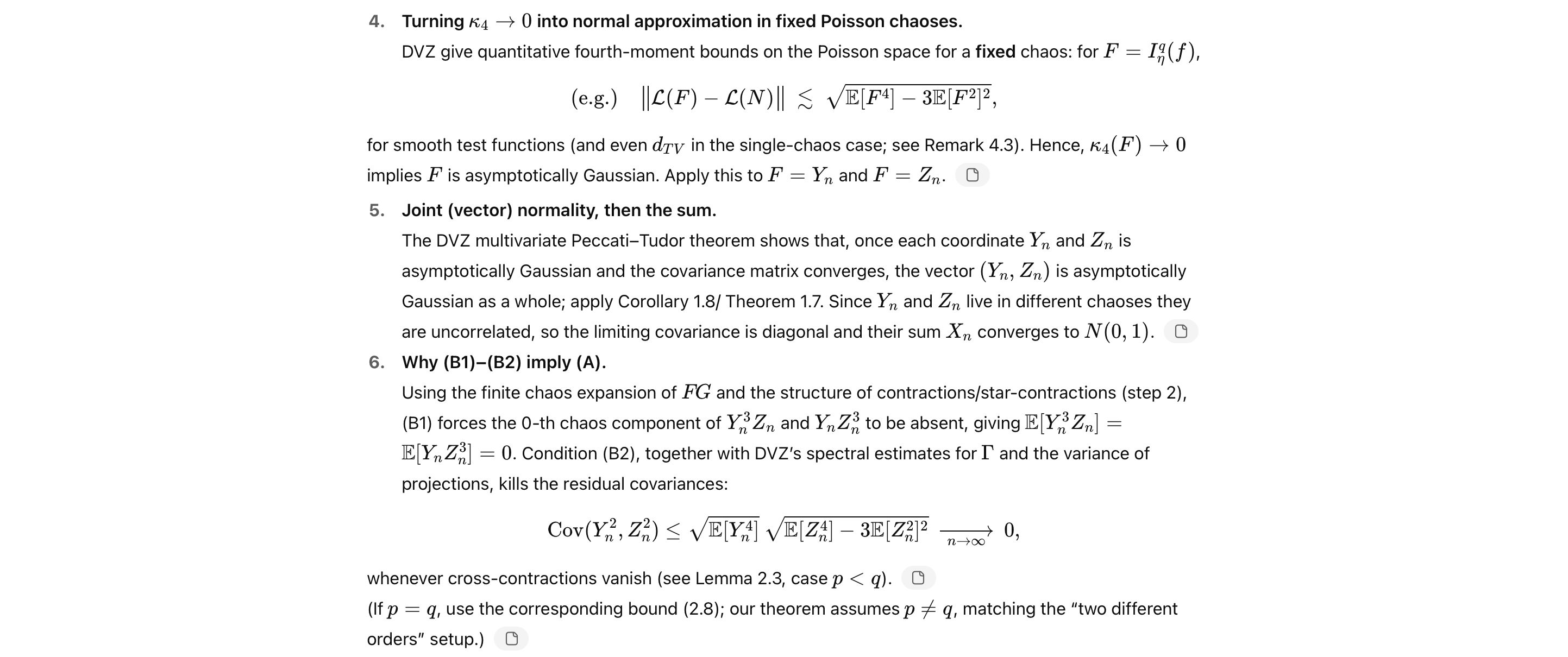}\par
  \medskip

\noindent
  \hspace*{-0.5em}%
 \!\!\!\!\!\!\!\!\!\!\!\!\!\!\!\!\!\!\!\!\!\!\!\!\!\!\!\!\!\!\!\!\!\!\!\!\!\!\!\!\!\!\!\!\!\!\!\!!\!\!\!\!\!\!\!\!\!\!\!\!\!\!
 \includegraphics[width=2\linewidth]{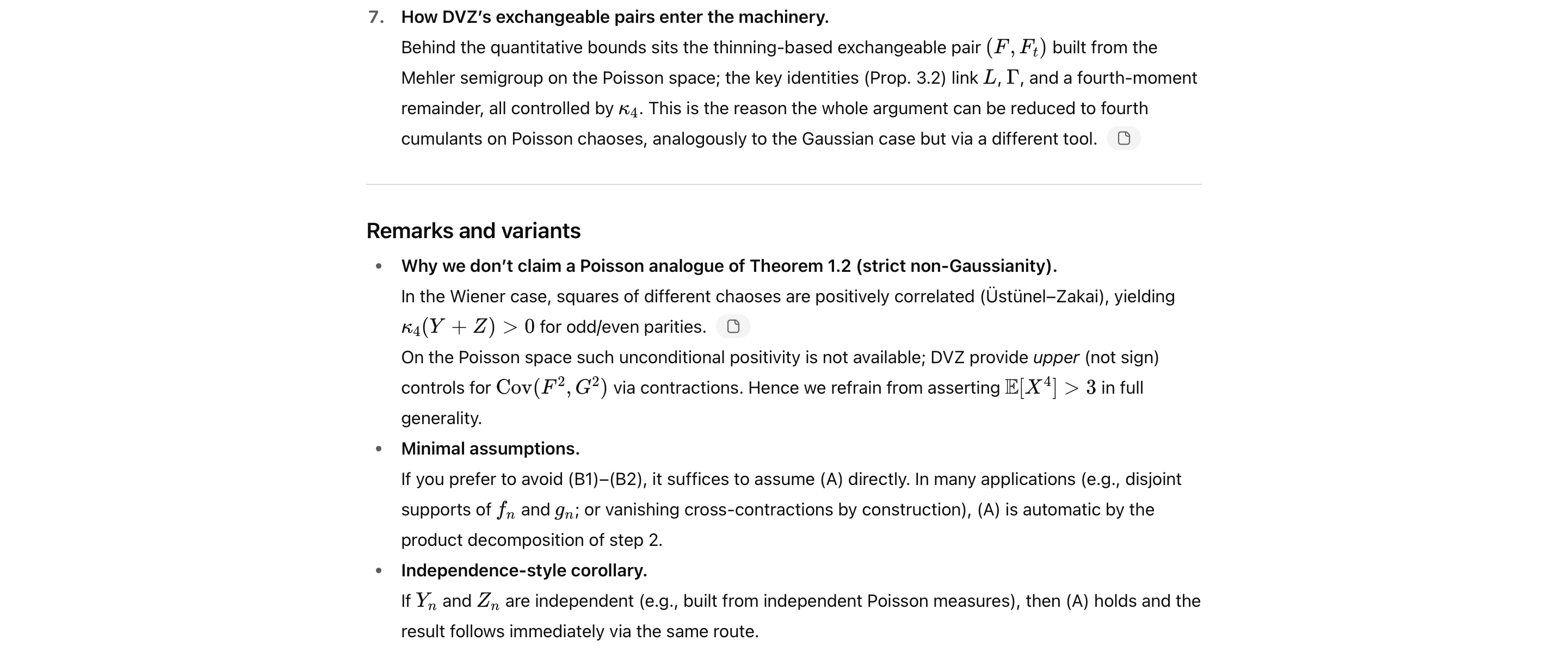}\par
  \medskip

\noindent
  \hspace*{-0.5em}%
 \!\!\!\!\!\!\!\!\!\!\!\!\!\!\!\!\!\!\!\!\!\!\!\!\!\!\!\!\!\!\!\!\!\!\!\!\!\!\!\!\!\!\!\!\!\!\!\!!\!\!\!\!\!\!\!\!\!\!\!\!\!\!
 \includegraphics[width=2\linewidth]{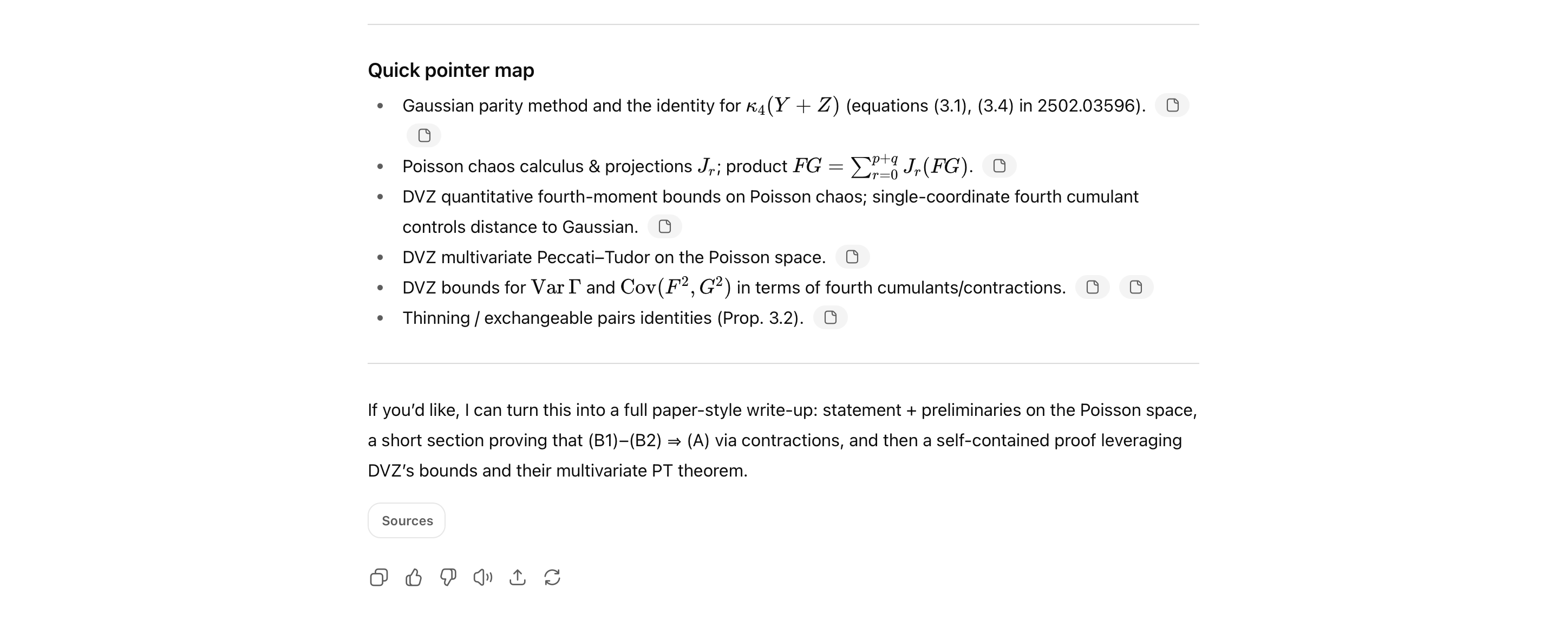}\par
  \medskip

\noindent
  \hspace*{-0.5em}%
 \!\!\!\!\!\!\!\!\!\!\!\!\!\!\!\!\!\!\!\!\!\!\!\!\!\!\!\!\!\!\!\!\!\!\!\!\!\!\!\!\!\!\!\!\!\!\!\!!\!\!\!\!\!\!\!\!\!\!\!\!\!\!
 \includegraphics[width=2\linewidth]{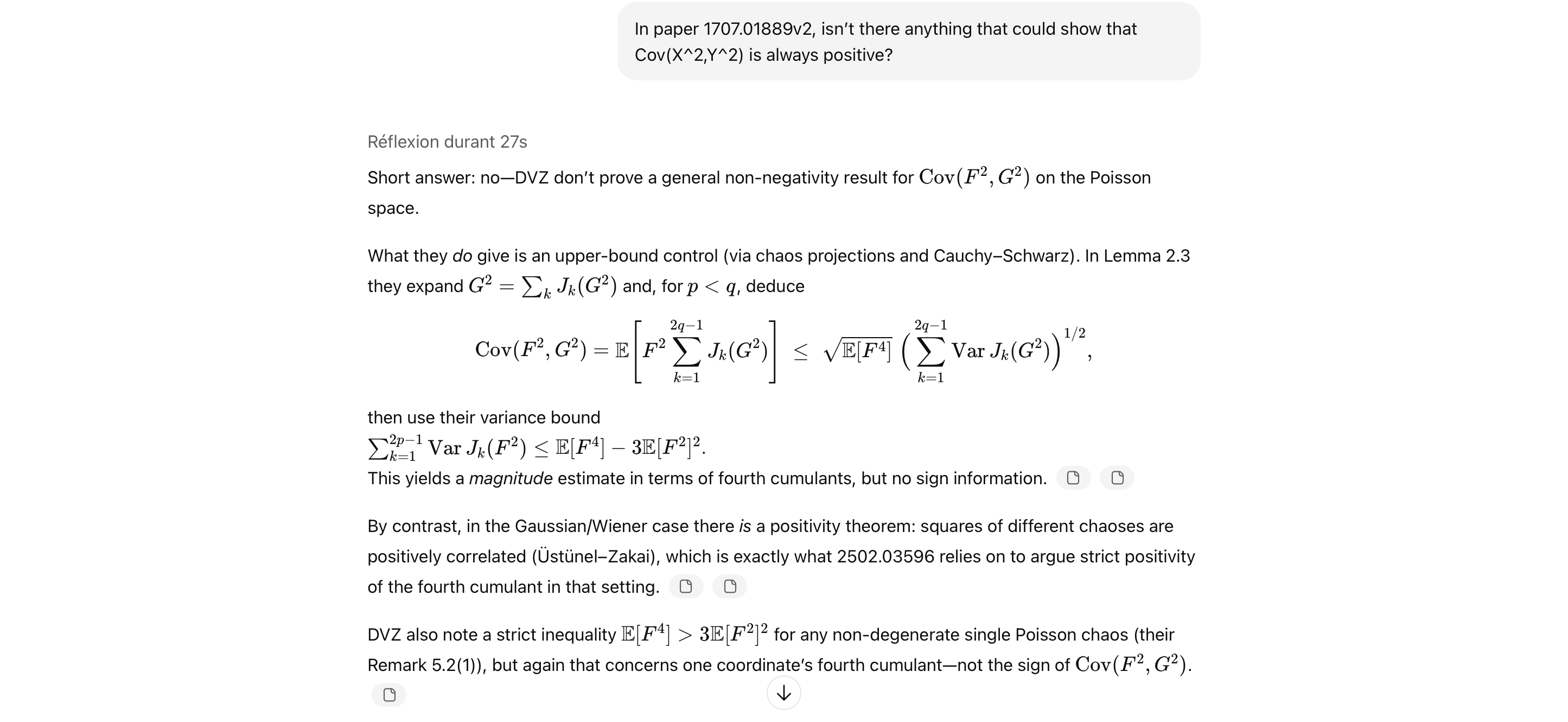}\par
  \medskip

\noindent
  \hspace*{-0.5em}%
 \!\!\!\!\!\!\!\!\!\!\!\!\!\!\!\!\!\!\!\!\!\!\!\!\!\!\!\!\!\!\!\!\!\!\!\!\!\!\!\!\!\!\!\!\!\!\!\!!\!\!\!\!\!\!\!\!\!\!\!\!\!\!
 \includegraphics[width=2\linewidth]{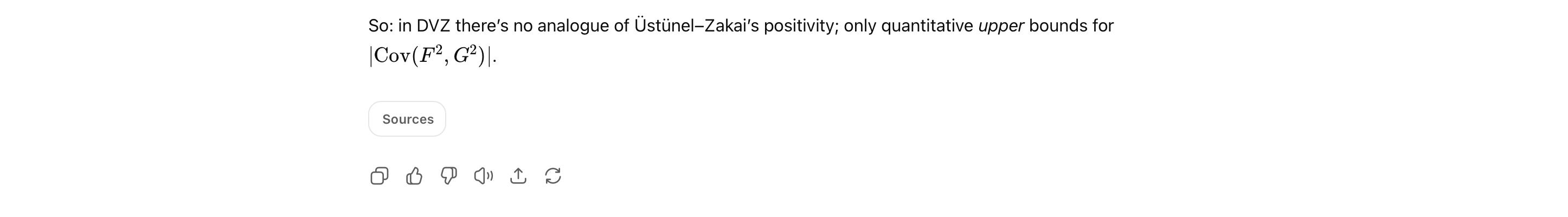}\par
  \medskip

\noindent
  \hspace*{-0.5em}%
 \!\!\!\!\!\!\!\!\!\!\!\!\!\!\!\!\!\!\!\!\!\!\!\!\!\!\!\!\!\!\!\!\!\!\!\!\!\!\!\!\!\!\!\!\!\!\!\!!\!\!\!\!\!\!\!\!\!\!\!\!\!\!
 \includegraphics[width=2\linewidth]{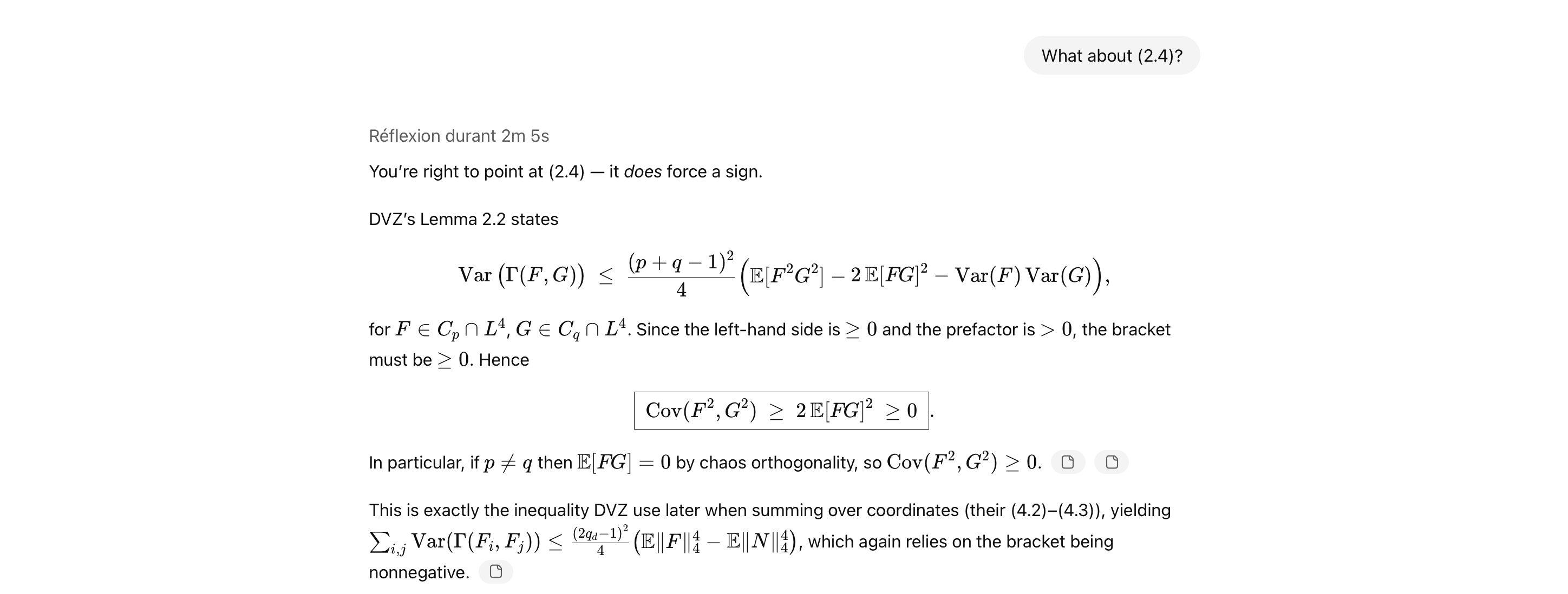}\par
  \medskip

\noindent
  \hspace*{-0.5em}%
 \!\!\!\!\!\!\!\!\!\!\!\!\!\!\!\!\!\!\!\!\!\!\!\!\!\!\!\!\!\!\!\!\!\!\!\!\!\!\!\!\!\!\!\!\!\!\!\!!\!\!\!\!\!\!\!\!\!\!\!\!\!\!
 \includegraphics[width=2\linewidth]{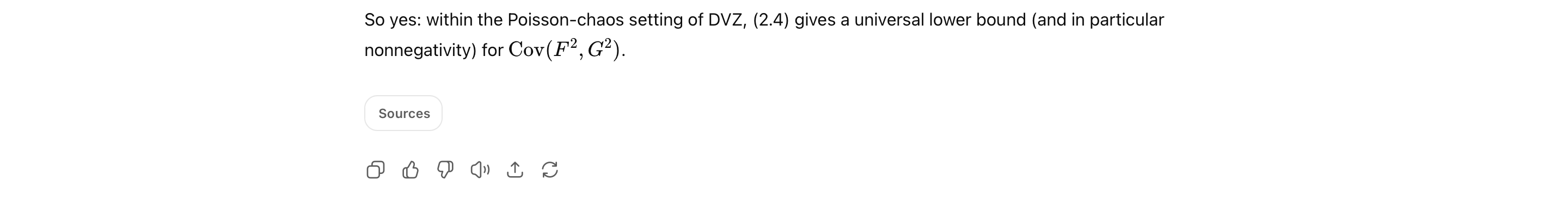}\par
  \medskip

\noindent
  \hspace*{-0.5em}%
 \!\!\!\!\!\!\!\!\!\!\!\!\!\!\!\!\!\!\!\!\!\!\!\!\!\!\!\!\!\!\!\!\!\!\!\!\!\!\!\!\!\!\!\!\!\!\!\!!\!\!\!\!\!\!\!\!\!\!\!\!\!\!
 \includegraphics[width=2\linewidth]{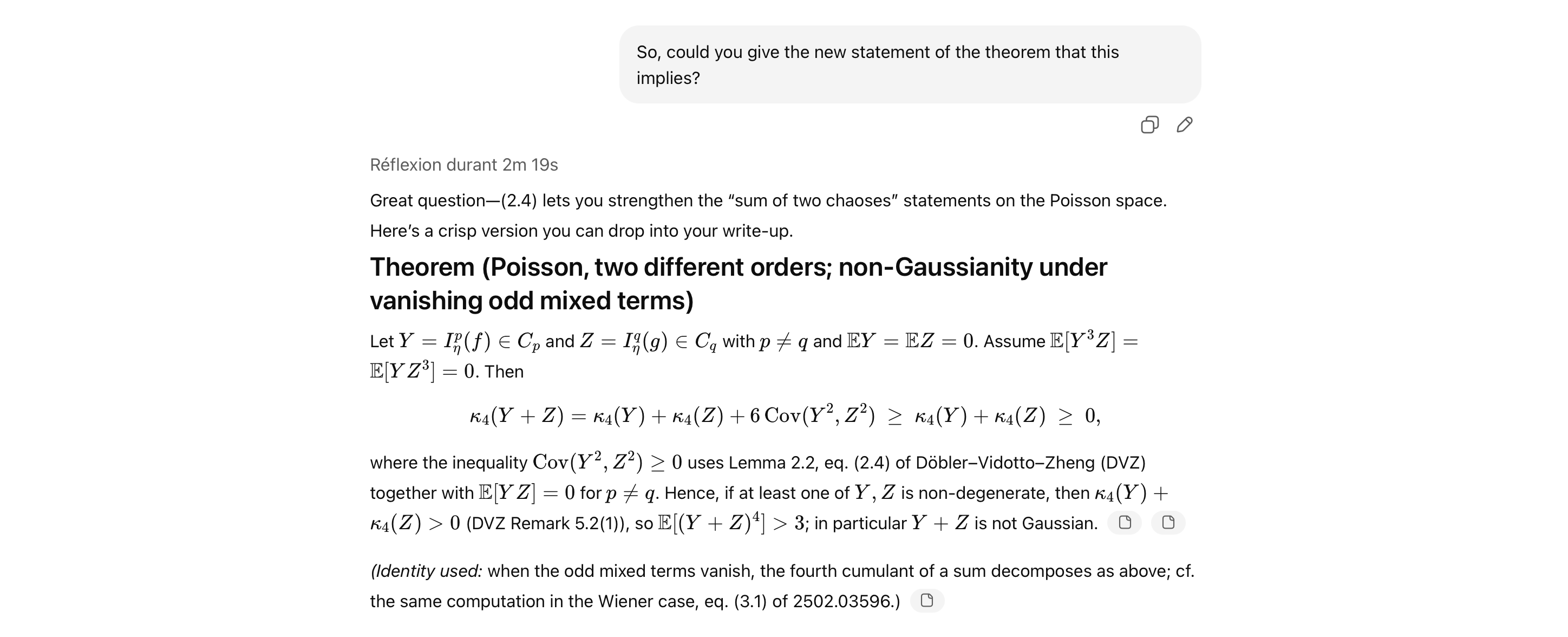}\par
  \medskip

\noindent
  \hspace*{-0.5em}%
 \!\!\!\!\!\!\!\!\!\!\!\!\!\!\!\!\!\!\!\!\!\!\!\!\!\!\!\!\!\!\!\!\!\!\!\!\!\!\!\!\!\!\!\!\!\!\!\!!\!\!\!\!\!\!\!\!\!\!\!\!\!\!
 \includegraphics[width=2\linewidth]{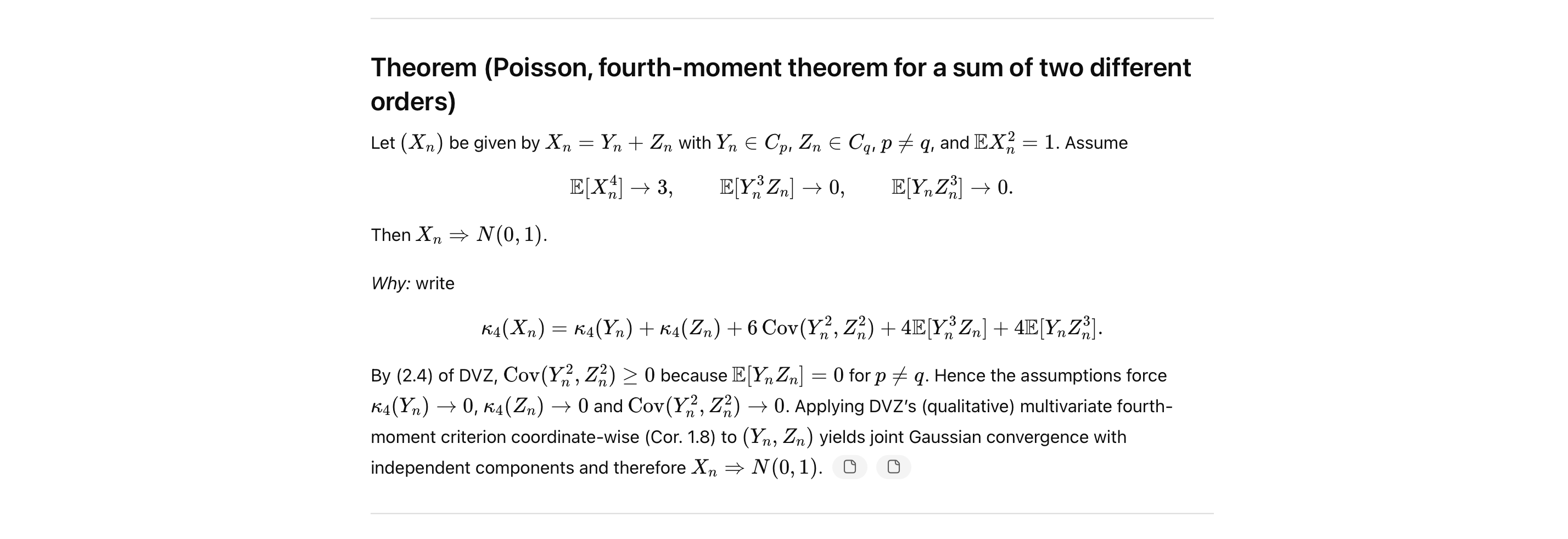}\par
  \medskip

\noindent
  \hspace*{-0.5em}%
 \!\!\!\!\!\!\!\!\!\!\!\!\!\!\!\!\!\!\!\!\!\!\!\!\!\!\!\!\!\!\!\!\!\!\!\!\!\!\!\!\!\!\!\!\!\!\!\!!\!\!\!\!\!\!\!\!\!\!\!\!\!\!
 \includegraphics[width=2\linewidth]{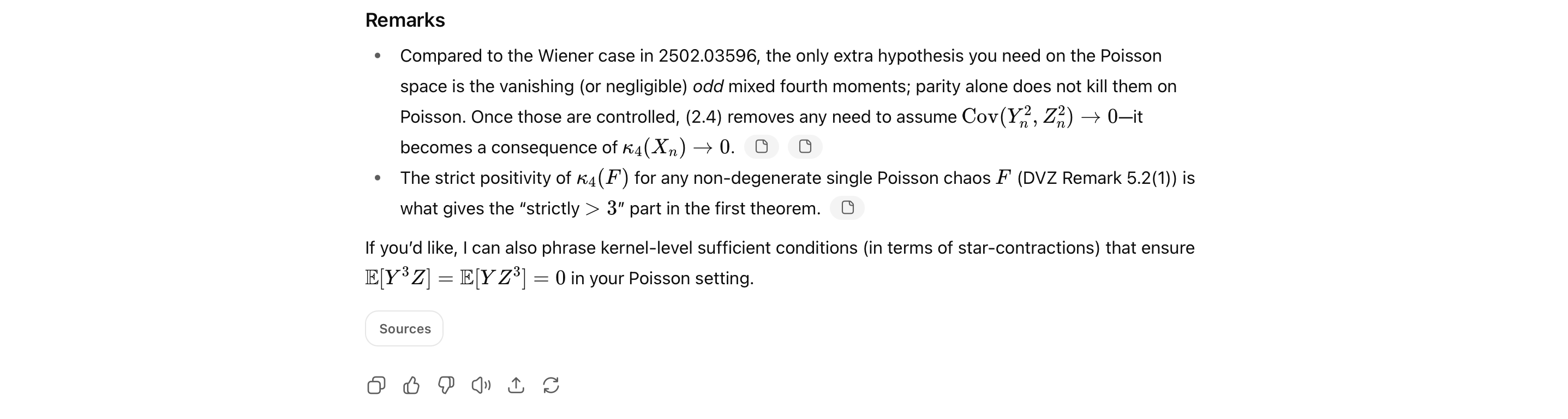}\par
  \medskip

\noindent
  \hspace*{-0.5em}%
 \!\!\!\!\!\!\!\!\!\!\!\!\!\!\!\!\!\!\!\!\!\!\!\!\!\!\!\!\!\!\!\!\!\!\!\!\!\!\!\!\!\!\!\!\!\!\!\!!\!\!\!\!\!\!\!\!\!\!\!\!\!\!
 \includegraphics[width=2\linewidth]{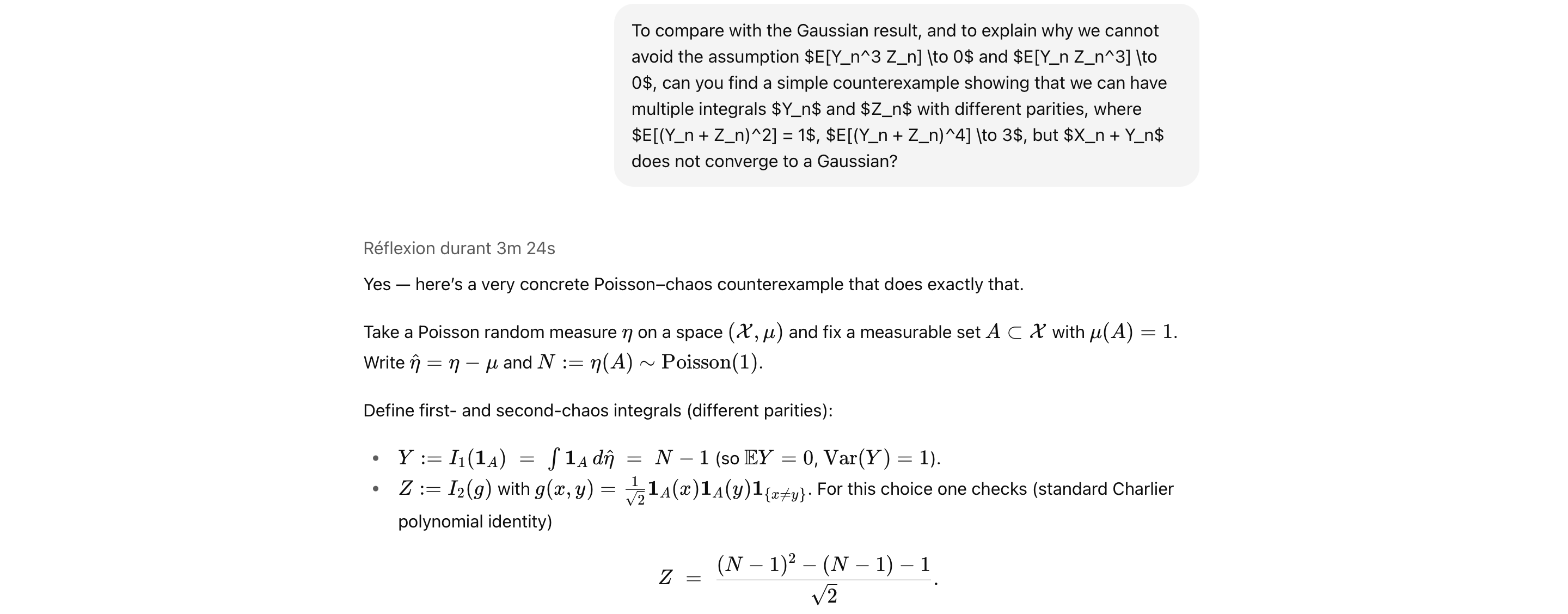}\par
  \medskip

\noindent
  \hspace*{-0.5em}%
 \!\!\!\!\!\!\!\!\!\!\!\!\!\!\!\!\!\!\!\!\!\!\!\!\!\!\!\!\!\!\!\!\!\!\!\!\!\!\!\!\!\!\!\!\!\!\!\!!\!\!\!\!\!\!\!\!\!\!\!\!\!\!
 \includegraphics[width=2\linewidth]{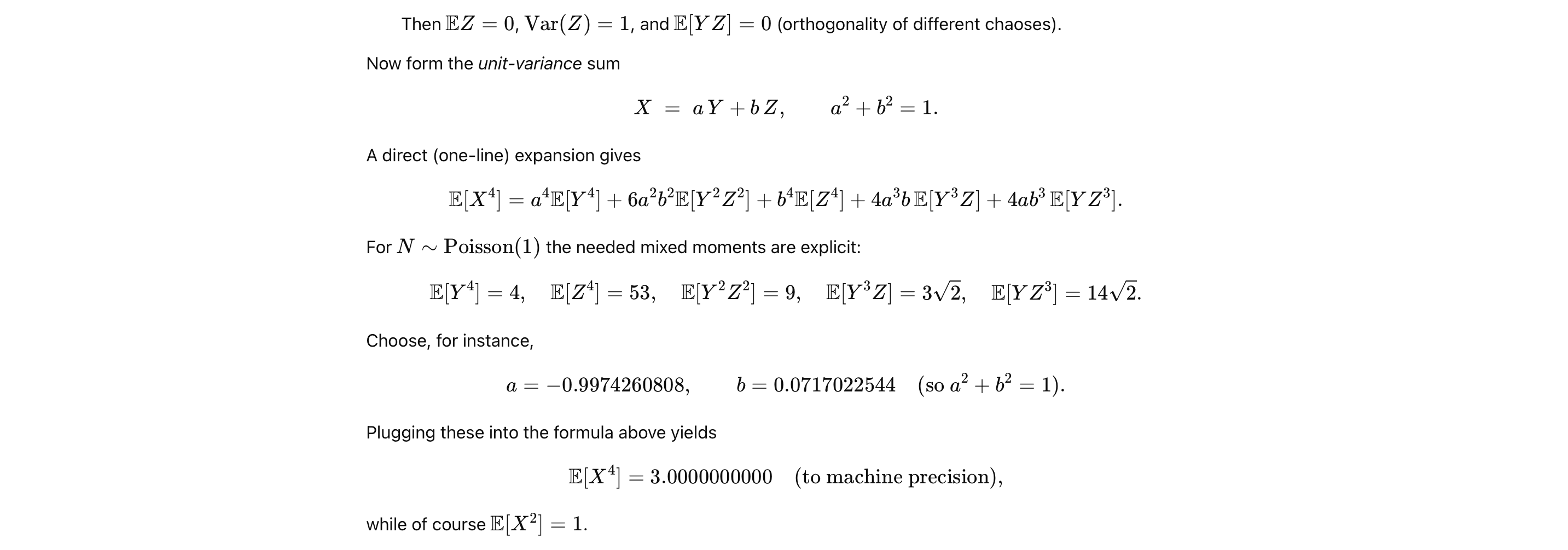}\par
  \medskip

\noindent
  \hspace*{-0.5em}%
 \!\!\!\!\!\!\!\!\!\!\!\!\!\!\!\!\!\!\!\!\!\!\!\!\!\!\!\!\!\!\!\!\!\!\!\!\!\!\!\!\!\!\!\!\!\!\!\!!\!\!\!\!\!\!\!\!\!\!\!\!\!\!
 \includegraphics[width=2\linewidth]{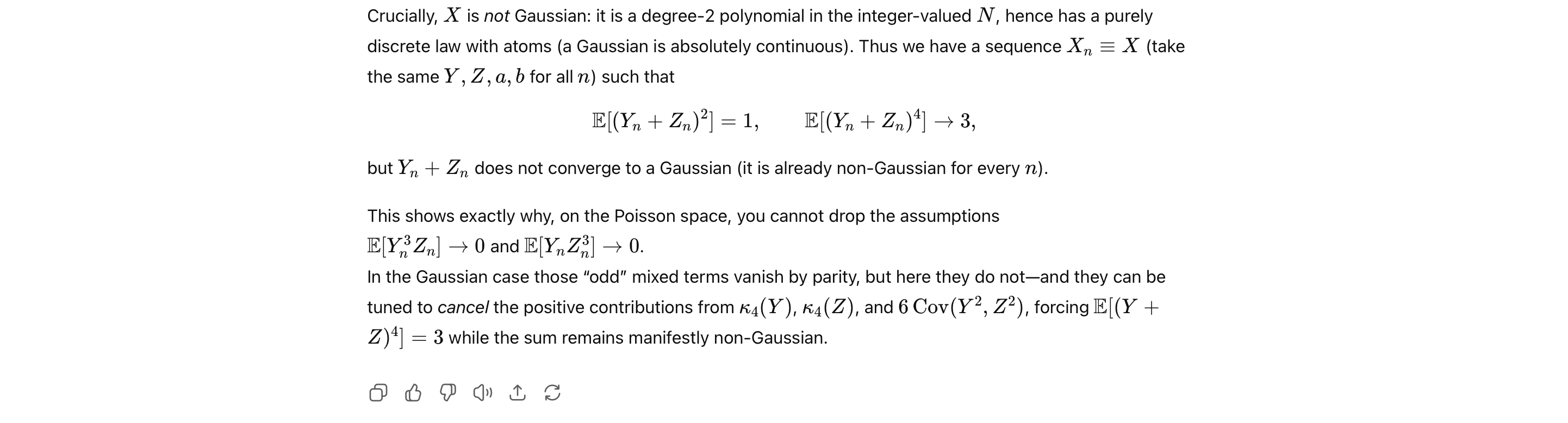}\par
  \medskip


\begin{thebibliography}{9}
\bibitem{basse} Basse-O’Connor, Kramer-Bang and Svendsen (2025). Fourth-Moment Theorems for Sums of Multiple Integrals. arXiv:2502.03596.
\bibitem{Bubeck2025Thread} \url{https://x.com/SebastienBubeck/status/1958198661139009862}
\bibitem{poisson} C. Döbler, A. Vidotto and G. Zheng (2018): Fourth moment theorems on the Poisson space in any dimension. Electron. J. Probab. 23: 1--27. arXiv:1707.01889
\bibitem{Nourdin2008} I. Nourdin (2008): Asymptotic behavior of weighted quadratic and cubic variations of fractional Brownian motion.  Ann. Probab. 36(6), 2159--2175. arXiv:0705.0570
\bibitem{Nourdin} I. Nourdin (2012): Lectures on Gaussian approximations with Malliavin calculus. Sém. Probab. XLV, 3--89. arXiv:1203.4147.
\bibitem{NourdinPeccati2012} I. Nourdin and G. Peccati (2012). Normal Approximations with Malliavin Calculus: From Stein's Method to Universality. Cambridge Univ. Press.
\bibitem{nourdin-rosinski} I. Nourdin and J. Rosinski (2014). Asymptotic independence of multiple Wiener-It\^o integrals and the resulting limit laws. Ann. Probab. 42(2), 497--526
\bibitem{nualart-peccati-2005} D. Nualart and G. Peccati (2005). Central limit theorems for sequences of multiple stochastic integrals. Ann. Probab. 33(1), 177--193.
\bibitem{Nualart2006} D. Nualart (2006). The Malliavin Calculus and Related Topics, 2nd ed. Springer.
\bibitem{Ryu} \url{https://nitter.net/ErnestRyu/status/1958408925864403068}
\end{thebibliography}
\end{document}